\documentclass[10pt,reqno]{amsart}
\usepackage{hyperref}
\usepackage{latexsym,amsmath}
\usepackage{centernot}
\usepackage{enumerate}
\usepackage{amsfonts}
\usepackage{amssymb}
\usepackage{latexsym}
\usepackage{fixmath}
\usepackage{yfonts}
 \usepackage[usenames,dvipsnames]{color}
 \usepackage[utf8]{inputenc}

\DeclareFontFamily{U}{mathb}{\hyphenchar\font45}
\DeclareFontShape{U}{mathb}{m}{n}{
<-6> mathb5 <6-7> mathb6 <7-8> mathb7
<8-9> mathb8 <9-10> mathb9
<10-12> mathb10 <12-> mathb12
}{}
\DeclareSymbolFont{mathb}{U}{mathb}{m}{n}
\DeclareMathSymbol{\llcurly}{\mathrel}{mathb}{"CE}
\DeclareMathSymbol{\ggcurly}{\mathrel}{mathb}{"CF}

\usepackage{graphicx}
 
\usepackage[T1]{fontenc}
\usepackage{fourier}
\usepackage{bbm}
\usepackage{color}
\usepackage{fullpage}
\usepackage{comment}

\usepackage{enumitem}
\usepackage{linearA}

\numberwithin{equation}{section}

\newcommand{\WP}{\mathrm{WP}}
\renewcommand{\vec}[1]{\boldsymbol{#1}}
\renewcommand{\subset}{\subseteq}

\newcommand\vD{\vec D}
\newcommand\MU{\vec\mu}

\newcommand\cC{\mathcal{C}}
\newcommand\cD{\mathcal{D}}
\newcommand\cG{\mathcal{G}}
\newcommand\cE{\mathcal{E}}
\newcommand\cS{\mathcal{S}}
\newcommand\cT{\mathcal{T}}

\newcommand\cM{\mathcal{M}}
\newcommand\cP{\mathcal{P}}

\newcommand\eps{\varepsilon}
\newcommand\GG{\mathbb{G}}

\newcommand\NN{\mathbb{N}}
\newcommand\Var{\mathrm{Var}}
\newcommand\Erw{\mathbb{E}}
\newcommand{\vecone}{\vec{1}}
\newcommand{\Po}{{\rm Po}}
\newcommand{\Bin}{{\rm Bin}}
\newcommand\bc[1]{\left({#1}\right)}
\newcommand\cbc[1]{\left\{{#1}\right\}}
\newcommand\brk[1]{\left[{#1}\right]}
\newcommand\RR{\mathbb{R}}
\newcommand{\Whp}{W.h.p.}
\newcommand{\whp}{w.h.p.}

\newcommand\pr{\mathbb{P}} 
\renewcommand\Pr{\pr} 
\newcommand\Thm{Theorem}

\newtheorem{definition}{Definition}[section]
\newtheorem{claim}[definition]{Claim}
\newtheorem{remark}[definition]{Remark}
\newtheorem{theorem}[definition]{Theorem}
\newtheorem{lemma}[definition]{Lemma}
\newtheorem{proposition}[definition]{Proposition}
\newtheorem{corollary}[definition]{Corollary}

\newtheorem{fact}[definition]{Fact}
\newtheorem{assumption}[definition]{Assumption}

\def\pr{{\mathbb P}}

\newcommand{\changebranch}{\textswab{T}}

\newcommand{\changeold}[1]{\sigma_{#1}}
\newcommand{\changenew}[1]{\tau_{#1}}
\newcommand{\changeup}[1]{\xi_{#1}}
\newcommand{\changeboth}[1]{\vec \sigma_{#1}}
\newcommand{\changes}{\vec \sigma}
\newcommand{\WPchanges}{\cE_{\mathrm{WP}}}
\newcommand{\markchanges}{\cE_{\mathrm{mark}}}

\newcommand{\messagehistory}[3]{\vec \mu_{#1 \to #2}(\le #3)}

\newcommand{\initialdistribution}{{Q}_0}

\newcommand{\potentialchanges}{\cP}

\newcommand{\wpf}{\varphi}
\newcommand{\distf}{\phi_{\varphi}}
\newcommand{\vdistf}[1]{\vec \phi_{\varphi}^{#1}}

\newcommand{\alphabet}{\Sigma}
\newcommand{\msgspace}{\cM}

\newcommand{\ms}[2]{\ensuremath{\left(\kern-.2em\left(\genfrac{}{}{0pt}{}{#1}{#2}\right)\kern-.2em\right)}}
\newcommand{\mset}[1]{\ensuremath{\left\{\kern-.1em\left\{#1\right\}\kern-.1em\right\}}}

\newcommand{\messgraphs}[1]{\cG_{#1}}
\newcommand{\histgraphs}[2]{\vec \cG_{#1}^{(#2)}}
\newcommand{\proj}[1]{\overline{#1}}

\newcommand{\close}[1]{\sim_{#1}}
\newcommand{\vclose}[1]{\approx_{#1}}

\newcommand{\messaget}[1]{$#1$-message}

\newcommand{\inmessages}{in-messages}
\newcommand{\inmessaget}[1]{$#1$-in-message}
\newcommand{\inmessagest}[1]{$#1$-in-messages}

\newcommand{\outmessages}{out-messages}
\newcommand{\outmessaget}[1]{$#1$-out-message}
\newcommand{\outmessagest}[1]{$#1$-out-messages}

\newcommand{\history}{history}
\newcommand{\histories}{histories}
\newcommand{\historyt}[1]{$#1$-history}
\newcommand{\historiest}[1]{$#1$-histories}

\newcommand{\instory}{in-story}
\newcommand{\instories}{in-stories}
\newcommand{\instoryt}[1]{$#1$-in-story}
\newcommand{\instoriest}[1]{$#1$-in-stories}

\newcommand{\outstory}{out-story}
\newcommand{\outstories}{out-stories}
\newcommand{\outstoryt}[1]{$#1$-out-story}

\newcommand{\inp}{input}
\newcommand{\inps}{inputs}
\newcommand{\inpt}[1]{$#1$-input}
\newcommand{\inpst}[1]{$#1$-inputs}

\newcommand{\story}{story}
\newcommand{\stories}{stories}
\newcommand{\storyt}[1]{$#1$-story}
\newcommand{\storiest}[1]{$#1$-stories}

\newcommand{\incompilation}{in-compilation}
\newcommand{\incompilations}{in-compilations}
\newcommand{\incompilationt}[1]{$#1$-in-compilation}
\newcommand{\incompilationst}[1]{$#1$-in-compilations}

\newcommand{\incompseq}{in-compilation sequence}

\newcommand{\incompseqt}[1]{$#1$-in-compilation sequence}

\newcommand{\compseq}{compilation sequence}

\newcommand{\mine}{snag}
\newcommand{\mines}{snags}
\newcommand{\landmine}{duplicate}
\newcommand{\seamine}{error}
\newcommand{\nuclearmine}{freak}
\newcommand{\shrapnel}{spurious}
\newcommand{\rubble}{defective}
\newcommand{\spray}{erroneous}
\newcommand{\fallout}{faulty}

\newcommand{\barm}{\overline{m}}
\newcommand{\voone}{\vom_1}
\newcommand{\votwo}{\vom_2}
\newcommand{\vom}{\vec \mu}

\newcommand{\wpg}{\mathbb{G}}
\newcommand{\apg}{\hat{\mathbb{G}}}
\newcommand{\wpm}{m}
\newcommand{\apm}{\hat m}
\newcommand{\wpr}{r}
\newcommand{\apr}{\hat r}

\newcommand{\degdist}{{\bf \mathcal{Z}}}
\newcommand{\degdistset}{{\bf \mathcal{Z}}}
\newcommand{\offdist}[2]{\mathcal{Y}_{#2#1}}
\newcommand{\MSet}[2]{\mathcal{M}\bc{ #1,#2}  }

\newcommand{\Matrdist}[2]{ #1\brk{#2} }
\newcommand{\abs}[1]{\left\vert #1\right\vert}
\newcommand{\Consist}[1]{\cC_{#1}}
\newcommand{\Compat}[1]{\cD_{#1}}

\newcommand{\muin}{\vec{\mu}_{\mathrm{in}}}
\newcommand{\muout}{\mu_{\mathrm{out}}}

\newcommand{\ball}{B}
\newcommand{\empNBdist}{\mathfrak{U}}
\newcommand{\SetErootedG}{ \mathcal{G}}
\newcommand{\dTV}{d_{\mathrm{TV}}}
\newcommand{\typedeg}{\vec d}
\newcommand{\typedegind}{d}
\newcommand{\Admit}{\mathcal{K}}

\newcommand{\Seq}{\mathrm{Se}}
\newcommand{\switch}[1]{\bar{#1}}
\newcommand{\norm}[1]{\left\lVert#1\right\rVert}
\newcommand{\conddist}[2]{#1\vert_{#2}}

\newcommand{\distphi}[2]{#1^{\bc{ #2}}}
\newcommand{\vdistphi}[2]{#1^{\bc{ \leq #2}}}
\newcommand{\viteratdist}[1]{\vdistphi{Q}{#1}}

\newcommand{\vprobvec}{\mathcal{N}}

\newcommand{\muchless}{\llcurly }

\begin{document}

\title{Warning Propagation: stability and subcriticality}

\author{Oliver Cooley, Joon~Lee, Jean B.~Ravelomanana}

\thanks{Jean B.~Ravelomanana is supported by DFG CO 646/4.\\
Oliver Cooley is supported by Austrian Science Fund (FWF): I3747}

\address{Oliver Cooley, {\tt cooley@math.tugraz.at}, Graz University of Technology, Institute of Discrete Mathematics, Steyrergasse 30, 8010 Graz, Austria}

\address{Joon Lee and Jean B. Ravelomanana, {\tt \{joon.lee,jean.ravelomanana\}@tu-dortmund.de}, TU Dortmund, Fakult\"at f\"ur Informatik, 12 Otto-Hahn-Strasse, Dortmund, 44227, Germany. }

\begin{abstract}
	Warning Propagation is a combinatorial message passing algorithm that unifies and generalises a wide variety of recursive combinatorial procedures.
	Special cases include the Unit Clause Propagation and Pure Literal algorithms for satisfiability as well as the peeling process for identifying the $k$-core of a random graph.
	Here we analyse Warning Propagation in full generality on a very general class of multi-type random graphs.
	We prove that under mild assumptions on the random graph model and the stability of the the message limit,
	Warning Propagation converges rapidly. 
	In effect, the analysis of the fixed point of the message passing process on a random graph reduces to
	analysing the process on a multi-type Galton-Watson tree.
	This result corroborates and generalises a heuristic first put forward by Pittel, Spencer and Wormald in their seminal $k$-core paper (JCTB 1996).
	\hfill[MSc: 05C80]

\end{abstract}

\maketitle

\section{Introduction}

\subsection{Motivation and contributions}

The study of combinatorial structures in random graphs is a huge field
encompassing a wide variety of different topics, and the techniques used to study them
are as plentiful and as varied as the topics themselves, but
there are common themes to be found in approaches in seemingly unrelated areas.
One such theme is the implementation of a discrete-time algorithm to
pinpoint the desired substructure. A classic example is Unit Clause Propagation,
an algorithm which traces implications in a Boolean satisfiability problem~\cite{Achlioptas,FS}.
If the formula contains unit clauses, i.e.\ clauses containing only one literal, the algorithm
sets the corresponding variable to the appropriate truth value. This clearly has
further knock-on effects: other clauses in which the variable appears with the same sign are now
automatically satisfied and can be deleted; but clauses in which the variable appears with the opposite sign
are effectively shortened, potentially giving rise to further unit clauses, and the process continues.
Ultimately, we may reach a contradiction or a satisfying assignment, or neither if the process stops
with all clauses containing at least two literals. In this case we can ``have a guess'',
assigning a random truth value to a random variable and continue the process.

Another quintessential example is the peeling process for the $k$-core, in which recursively
vertices of degree at most $k-1$ are deleted from the graph until what remains is the
(possibly empty) $k$-core (see e.g.~\cite{Pittel,MolloyCores}).
Further examples include the study of sparse random matrices, the freezing phase transition in random constraint satisfaction problems,
bootstrap percolation or decoding low-density parity check codes~\cite{Barriers,COFKR,DuboisMandler,Gallager,Molloy,RichardsonUrbanke}.

Warning Propagation is a a message passing scheme that provides a unified framework
for such recursive processes~\cite{MM}. Roughly speaking, the scheme sends messages along
edges of a graph which are then recursively updated: the messages that a vertex sends depends
on the messages that it receives from its neighbours according to some update rule.
The semantics of the messages and the choice of update rule is fundamentally dependent on the
particular problem to which the scheme is applied: the messages may indicate truth values
of variables in a satisfiability formula, for example, or membership of the $k$-core.
To understand the combinatorial substructures under consideration, we need to understand
the fixed points of the corresponding recursive algorithms, or equivalently the fixed points of the
appropriate instances of Warning Propagation.

There have been many different approaches to analysing such recursive processes
using a variety of different techniques. One classical tool is the differential equations method~\cite{Wormald},
which was used in the seminal $k$-core paper of Pittel, Spencer and~Wormald~\cite{Pittel}
as well as in the analysis of Unit Clause Propagation~\cite{Achlioptas}.
Other approaches include branching processes~\cite{Riordan}, enumerative methods~\cite{COCKS2}, or birth-death processes~\cite{JansonLuczak,JansonLuczak2}.

However, despite their very different appearances, these approaches all share a common feature:
in one way or another, they show that the recursive process converges quickly to its fixed point.
In other words, the final outcome of the process can be approximated arbitrarily well by running only
a bounded number of rounds of the recursive process. Equivalently, in each of these particular instances,
the Warning Propagation scheme converges quickly.

In this paper we analyse Warning Propagation in full generality on a very general multi-type model of random graphs.
Special cases of this model include not just the Erd\H{o}s-R\'enyi binomial random graph model $G\bc{n,p}$
and its $k$-partite analogues, but also the stochastic block model, random regular graphs or indeed random graphs with a prescribed degree sequence,
and factor graphs of random hypergraphs.
We prove that under mild, easy-to-check assumptions Warning Propagation converges rapidly.
Not only does this result confirm the heuristic that running Warning Propagation for a bounded number of rounds
suffices to approximate its ultimate fixed point arbitrarily well,
our result also identifies the essential reason for this behaviour.
More precisely, after a large but bounded number of steps,
the subsequent knock-on effect of a single change can be modelled by a branching process;
we demonstrate that a mild stability assumption guarantees that this branching process is subcritical.
The upshot is that late changes in the process will ultimately fizzle out rather than triggering
a cascade of further effects.

Apart from re-proving known results in a new, unified way, the main results of this paper facilitate new applications of Warning Propagation.
Indeed, to analyse any specific recursive process that can be translated into the formalism of \Thm~\ref{thm:main} below one just needs to investigate the recursion on a multi-type Galton-Watson tree that mimics the local structure of the respective random graph model.
Typically this task boils down to a mundane fixed point problem in Euclidean space.
\Thm~\ref{thm:main} thus enables an easy and accurate analysis of generic recursive processes on random structures.
A concrete example that actually inspired this work was our need to study a recursive process that arises in the context of random matrix theory~\cite{we}.

\subsection{Random graph model}

Our goal is to study warning propagation on a random graph $\wpg$, which may be chosen from a wide variety of different models,
and which we first describe briefly and informally---the formal requirements on $\wpg$ are introduced in Section~\ref{sec:assump},
specifically in Assumption~\ref{AssumeNew}.

We will assume that the vertices of $\wpg$ are of \emph{types} $1,\ldots,k$ for some fixed integer $k$; we denote by $V_i$ the set
of vertices of type $i$ for $i\in \brk{k}$ and set $n_i:=\abs{V_i}$. The $n_i$ need not be deterministically fixed, but may themselves be random
variables depending on an implicit parameter $n\in \NN$ which tends to infinity, and in particular all asymptotics are with respect to $n$
unless otherwise specified.
Vertices of different types may exhibit very different behaviour, but vertices of the same type should behave according to the
same random distribution.
More specifically, for a vertex $v \in V_i$ the (asymptotic) distribution 
of the numbers of neighbours of each type $j \in \brk{k}$ will be described by $\degdist_i$, which is
a probability distribution on $\NN_0^k$, the set of sequences of natural numbers of length $k$; 
the $j$-th entry of $\degdist_i$ describes the numbers of neighbours of type $j$. This will be introduced more formally in Section~\ref{sec:distfp}

To give a concrete example,
if we were to study simply $G\bc{n,d/n}$ for some fixed constant $d$, we would set
$k=1$ and $n_1=n$, and each vertex would have $\Po\bc{d}$ neighbours of type $1$.
For random $d$-regular graphs, we would also have $k=1$ and $n_1=n$, but now the number of neighbours would be deterministically $d$
(i.e.\ the random distribution would be entirely concentrated on $d$).

A slightly more complex example is random $d$-SAT with $n$ variables and $m$ clauses of size $d$.
The standard way of representing an instance of the problem is to have vertex classes $V_1,V_2$ representing the variables and the clauses respectively,
with an edge between a variable $v$ and a clause $A$ if $v$ appears in $A$. Furthermore, the edge is coloured depending
on whether $v$ is negated in $A$ or not. However, since we do not allow for edges of different types, we must represent this differently.
This can be done by adding two further classes $V_3,V_4$ and subdividing an edge $vA$ with a vertex of type~3
if $v$ is unnegated in $A$ and of type~4 otherwise.
Then a vertex of $V_1$, representing a variable, would have $\Po\bc{\frac{dm}{2n}}$ neighbours of type~3 and similarly and independently of type~4;
a vertex of $V_2$, representing a clause, would have $X \sim \Bin\bc{d,1/2}$ neighbours of type~3 and $d-X$ neighbours of type~4;
while vertices of $V_3,V_4$ would each
have precisely one neighbour each of types~1 and~2.

We will have various relatively loose restrictions on the graph model $\wpg$ which are required during the proof, see Section~\ref{sec:assump} for the full list.
Informally, we require $\wpg$ to satisfy four conditions with high probability, namely:
\begin{itemize}
\item The vertex classes have the same order of magnitude and not too large variance.
\item The graph $\wpg$ is uniformly random given its type-degree sequence.
\item There are few vertices of high degree.
\item The local structure is described by the $\cT_i\bc{\degdist_1,\ldots,\degdist_k}$.
\end{itemize}
Here we note in particular that we require each $V_i$ to have bounded average degree.

\subsection{Warning propagation}
In this section we formally introduce the Warning Propagation (WP) message passing scheme and its application to random graphs.
Applied to a graph $G$, Warning Propagation will associate two directed messages $\mu_{v\to w},\mu_{w \to v}$ with each edge $vw$ of $G$.
These messages take values in a finite alphabet $\alphabet$.
Hence, let $\msgspace\bc{G}$ be the set of all vectors $\bc{\mu_{v\to w}}_{\bc{v,w}\in V\bc{G}^2:vw\in E\bc{G}} \in \alphabet^{2\abs{E\bc{G}}}$.
The messages get updated in parallel according to some fixed rule.
To formalise this, for $d \in \NN$ let $\ms\alphabet d$ be the set of all $d$-ary multisets with elements from $\alphabet$ and let
\begin{align}\label{eqUpdate}
	\wpf:\bigcup_{d\geq0}\ms\alphabet d\to\alphabet
\end{align}
be an \emph{update rule} that, given any multiset of input messages, computes an output message.
Then we define the Warning Propagation operator on $G$ by
\begin{align*}
	\WP_G&:\msgspace\bc{G}\to\msgspace\bc{G},&&\mu=\bc{\mu_{v\to w}}_{vw}\mapsto\bc{\wpf\bc{\mset{\mu_{u\to v}:{uv\in E\bc{G},u\neq w}}}}_{vw},
\end{align*}
where $\mset{a_1,\ldots,a_k}$ denotes the multiset whose elements (with multiplicity) are $a_1,\ldots,a_k$.

In words, to update the message from $v$ to $w$ we apply the update rule~$\wpf$ to the messages that $v$ receives from all its {\em other} neighbours $u\neq w$.

To give some examples of concrete instances, when studying the $k$-core the messages would typically be $0$ or $1$,
and the update rule would be defined by $\wpf\bc{A} = \vecone\cbc{\sum_{a \in A}a \ge k-1}$, i.e.\ a vertex sends a message of~$1$
to a neighbour iff it receives at least $k-1$ messages of~$1$ from its other neighbours.
At the end of the process, the $k$-core consists of precisely those vertices which receive at least $k$ messages
of $1$ from their neighbours.
Alternatively, in a constraint satisfaction problem,
the message from a variable to a constraint may
indicate that the variable is frozen to a specific value due to its other constraints,
while the message from a constraint to a variable indicates
whether that constraint requires the variable to take a specific value.

Let us note that in many applications, the obvious approach would be to define the WP scheme with different
update rules $\wpf_1,\ldots,\wpf_k$ for each type of vertex, or indeed where the update rule takes account of which
type of vertex each message was received from. While this would be entirely natural,
it would lead to some significant notational complexities later on. We therefore adopt
an alternative approach: the messages of the alphabet $\alphabet$ will, in particular,
encode the types of the source and target vertices, and we can therefore make do with a single update function
which receives this information and takes account of it. Of course, this means that along
a particular directed edge, many messages from $\alphabet$ are automatically disqualified
from appearing because they encode the wrong source and target types.
Indeed, at a particular vertex all incoming messages must encode the same appropriate target type,
and therefore many multisets of messages can never arise as inputs of the update function.
On the other hand,
the major benefit of this approach is that much of the notational complexity of the problem is subsumed
into the alphabet $\alphabet$ and the update function $\wpf$. This will be discussed more formally in Sections~\ref{sec:prereq}, and~\ref{sec:apg}.

In most applications of Warning Propagation the update rule~\eqref{eqUpdate} enjoys a monotonicity property which ensures that for any graph $G$
and for any initialisation $\mu^{\bc{0}}\in\msgspace\bc{G}$ the pointwise limit $\WP_G^*\bc{\mu^{\bc{0}}} :=\lim_{t\to\infty}\WP_G^t\bc{\mu^{\bc{0}}}$ exists,
although in general monotonicity is
not a necessary prerequisite for such a limit to exist.
If it does, then clearly this limit is a fixed point of the Warning Propagation operator.

Our goal is to study the fixed points of $\WP$ and, particularly, the rate of convergence on
the random graph $\wpg$.
We will assume that locally $\wpg$ has the structure of a multi-type Galton-Watson tree.
We will prove that under mild assumptions on the update rule, the $\WP$ fixed point can be characterised in terms of this local structure only.
To this end we need to define a suitable notion of a $\WP$ fixed point on a random tree.
At this point we could consider the space of (possibly infinite) trees with $\WP$ messages, define a measure on this space
and consider the action that the WP operator induces. Fortunately, the recursive
nature of the Galton-Watson tree allows us to bypass this complexity.
Specifically, our fixed point will just be a collection of probability distributions on $\alphabet$, one for each possible type of directed edge,
such that if the children of a vertex $v$ in the tree send messages independently according to these distributions,
then the message from $v$ to its own parent will also have the appropriate distribution from the collection.
The collection of distributions can be conveniently expressed in matrix form.
For a matrix $M$, we denote by $\Matrdist{M}{i,j}$ the entry at position $\bc{i,j}$ in the matrix
and by $\Matrdist{M}{i}$ the $i$-th row $\bc{\Matrdist{M}{i,j}}_{j \in \brk{k}}$.
\footnote{We avoid the usual $M_{ij}$ index notation since this will clash with other subscripts later on.}

\begin{definition} \label{def:probaonMess}
	Given a set $S$, a \emph{probability distribution matrix} on $S$ is a $k\times k$ matrix $Q$ in which each entry $\Matrdist{Q}{i,j}$ of $Q$ is a probability distribution on $S$. 
\end{definition} 

The intuition is that the entry  $\Matrdist{Q}{i,j}$ should model the probability
distribution of the message along an edge from a vertex of type $i$ to a vertex of type $j$.
Heuristically, the incoming messages at a vertex will be more or less independent of each other; short-range correlations
can only arise because of short cycles, of which there are very few in the sparse regime, while long-range correlations should be weak if they exist at all.
We will certainly \emph{initialise} the messages independently.

\begin{definition} \label{def:initial}
For a graph $G$ and a probability distribution matrix $Q$ on $\alphabet$, we refer to \emph{initialising messages in $G$ according to $Q$}
to mean that we initialise the message $\mu_{u \to v}\bc{0}$ for each directed edge $\bc{u,v}$ independently at random according to $\Matrdist{Q}{i,j}$, where $i$ and $j$ are the types of $u$ and $v$ respectively.
\end{definition}

In many applications, the initialisation of the messages is actually deterministic, i.e., each entry of
$Q$ is concentrated on a single element of $\alphabet$, but there are certainly
situations in which it is important to initialise randomly.

Given the local structure of the random graph model $\wpg$ as described by a multi-type Galton-Watson tree,
we can compute the asymptotic effect
of the warning propagation update rules on the probability distribution matrix: for a directed edge $vw$ of type $\bc{i,j}$,
we consider the other neighbours of $v$ with their types according to the local structure, generate messages independently according
to the current probability distribution matrix and compute the updated message along $vw$. Since the generation of
neighbours and of messages was random, the updated $vw$ message is also random and its distribution gives
the corresponding entry of the updated matrix. Repeating this for all $i,j \in \brk{k}$ gives the updated matrix.
This process is described more formally in Section~\ref{sec:distfp}.

With this notion of updating probability distribution matrices, we can consider the \emph{limit} of an initially
chosen matrix $\initialdistribution$. More specifically, we will need the existence of a \emph{stable WP limit},
meaning that the update function is a contraction in the neighbourhood of the limit with respect to an appropriate metric.
Again, formal details are given in Section~\ref{sec:distfp}.

\subsection{Main result}

Given a probability distribution matrix $\initialdistribution$ on $\alphabet$, we ask how quickly Warning Propagation will converge on $\wpg$ from a random initialisation according to $\initialdistribution$.

We will use $\WP^t_{v\to w}\bc{\mu^{\bc{0}}}$ to denote the message from $v$ to $w$ in $\wpg$ after $t$ iterations of $\WP_\wpg$
with initialisation~$\mu^{\bc{0}}$. Note that the graph $\wpg$ is implicit in this notation.

\begin{theorem}\label{thm:main}
Let $\wpg$ be a random graph model satisfying Assumption~\ref{AssumeNew}
and let $P, \initialdistribution$ be probability distributions on $\alphabet$ such that $P$ is the stable WP limit of $\initialdistribution$.
Then for any $\delta>0$ there exists $t_0=t_0\bc{\delta, \degdist,\wpf,\initialdistribution}$ such that the following is true.

Suppose that $\mu^{\bc{0}}\in \msgspace\bc{\wpg}$ is an initialisation
according to $\initialdistribution$.
Then \whp\ for all $t\ge t_0$ we have
\begin{align*}
\sum_{v,w:vw\in E\bc{\wpg}}\vecone\cbc{\WP^t_{v\to w}\bc{\mu^{\bc{0}}}\neq\WP^{t_0}_{v\to w}\bc{\mu^{\bc{0}}}}<\delta n.
\end{align*}
\end{theorem}
In other words, the WP messages at any time $t\ge t_0$ are identical to those at time $t_0$ except on a set of at most $\delta n$ directed edges.
Thus \Thm~\ref{thm:main} shows that under a mild stability condition Warning Propagation converges rapidly.
Crucially, the number $t_0$ of steps before Warning Propagation stabilises does not depend on the underlying parameter $n$,
or even on the exact nature of the graph model $\wpg$,
but only on the desired accuracy $\delta$,
the degree distribution $\degdist$, the Warning Propagation update rule $\wpf$ and the initial distribution $\initialdistribution$.

\subsection{Discussion and related work}
\Thm~\ref{thm:main} implies a number of results that were previously derived by separate arguments.
For instance, the theorem directly implies the main result from~\cite{Pittel} on the $k$-core in random graphs.
Specifically, the theorem yields the threshold for the emergence of the $k$-core threshold as well as the typical number of vertices and edges in the core (in a law of large numbers sense).
Of course, several alternative proofs of (and extensions of) this result, some attributed as simple, exist~\cite{Cooper,Darling,Fernholz1,Fernholz2,JansonLuczak,Kim,MolloyCores,Riordan},
but here we obtain this result as an application of a more general theorem.

Since our model also covers multi-type graphs, it enables a systematic approach to the freezing phenomenon in
random constraint satisfaction problems~\cite{MM,Molloy,MR},
as well as to hypergraph analogues of the core problem~\cite{CKZ21,JansonLuczak,Kim,MolloyCores,Pittel,Riordan,Skubch}
by considering the factor graph.

The specific application that led us to investigate Warning Propagation in general deals with random matrix theory~\cite{we}.
In that context Warning Propagation or equivalent constructions have been applied extensively~\cite{AchlioptasMolloy,DuboisMandler,Ibrahimi,MM}.
Technically the approach that is most similar to the present proof strategy is that of Ibrahimi, Kanoria, Kraning and Montanari~\cite{Ibrahimi}, who use an argument based on local weak convergence.

\subsection{Proof outline}

A fundamental aspect of the proof is that we do not analyse $\WP$ directly on $\wpg$
and consider its effect after $t_0$ iterations,
but instead define an alternative random model
$\apg_{t_0}$ (see Definition~\ref{def:apg}): Rather than generating the edges of the graph and then computing messages,
this random model first generates half-edges with messages, and then matches up the half-edges
in a consistent way. Thus in particular the messages are known a priori.
The key point is that the two models are very similar (Lemma~\ref{lem:contiguity}).

Among other things, it follows from this approximation that very few changes will be made when moving from
$\WP_\wpg^{t_0-1}\bc{\mu^{\bc{0}}}$ to $\WP_\wpg^{t_0}\bc{\mu^{\bc{0}}}$, but in principle these few changes could
cause cascade effects later on.
To rule this out we define a branching process $\changebranch$ which approximates
the subsequent effects of a single change at time $t_0$.
The crucial observation is that the stability of the distributional fixed point $P$ implies that this
branching process is subcritical (Proposition~\ref{prop:subcritical}),
and is therefore likely to die out quickly.
Together with the fact that very few changes are made at step $t_0$,
this ultimately implies that there will be few subsequent changes.

\subsection{Paper overview}

The remainder of the paper is arranged as follows.
In Section~\ref{sec:prereq} we formally introduce the notation, terminology and
assumptions on the model~$\wpg$ which appear in the statement of Theorem~\ref{thm:main}
and throughout the paper.
In Section~\ref{sec:apg}
we define the $\apg_{t_0}$ model and introduce Lemma~\ref{lem:contiguity},
which states that this model is a good approximation for Warning Propagation on $\wpg$.
In Section~\ref{sec:prelim} we present various preliminary results
that will be used in later proofs.
In Section~\ref{sec:contiguity} we go on to prove Lemma~\ref{lem:contiguity}.

In Section~\ref{sec:subcritical} we introduce the branching process $\changebranch$
and prove that it is subcritical. In Section~\ref{sec:beyond} we then draw together the results
of previous sections to prove that after $t_0$ iterations of $\WP$, very few further changes
will be made, and thus prove Theorem~\ref{thm:main}.

\section{Prerequisites}\label{sec:prereq}

In this section we formally define some of the notions required for the statement of Theorem~\ref{thm:main},
as well as introducing the assumptions that we require the model $\wpg$ to satisfy. For a set $S$, we will denote
by $\cP\bc{S}$ the space of probability distributions on $S$. We will occasionally abuse notation by conflating a
random variable with its probability distribution, and using the same notation to refer to both.

\subsection{Distributional fixed points}\label{sec:distfp}

\begin{definition}\label{def:admissible}
For each $i\in\brk{k}$, let $\degdist_{i}\in \cP\bc{\NN_0^k}$.
For $j \in \brk{k}$, denote by $\degdist_{ij}$ the  marginal distributions of $\degdist_{i}$ on the $j$-th entry. 
  We say that $\bc{i,j}\in\brk{k}^2$ is an \emph{admissible} pair if \ $\Pr \bc{
 	\degdist_{ij} \geq 1 } \neq  0$, and denote by $\Admit=\Admit\bc{\degdist_1,\ldots,\degdist_k}$ the set of admissible pairs.
\end{definition}

Intuitively, the $\degdist_i$ will describe the local structure of the random input graph $\wpg$,
in the sense that the distribution of the neighbours with types of a vertex $v \in V_i$ will be approximately $\degdist_i$
(see Definition~\ref{def:GWproc} later).
Therefore
the admissible pairs describe precisely those pairs of classes $V_i$ and $V_j$ between which we expect some edges
to exist. In particular, if the $\degdist_i$ accurately describe the local structure, then $\bc{i,j}$ is admissible if and only if $\bc{j,i}$ is also admissible.

Note, however, that if we aim to analyse the message along a directed edge from $v \in V_i$ to $w\in V_j$,
we need to know about the distribution of the \emph{other} neighbours of $v$, and cannot simply draw from $\degdist_i$
because we already have one guaranteed neighbour of type $j$, which may affect the distribution. This motivates the following definition.

\begin{definition} \label{def:offspringdist}
Let $\degdist_1,\ldots,\degdist_k\in \cP\bc{\NN_0^k}$.
For each $\bc{i,j}\in\Admit$, define $\offdist{j}{i} = \offdist{j}{i}\bc{\degdist_i} \in \cP\bc{\NN_0^k}$ to
be the probability distribution such that for $\bc{a_1, \ldots, a_k} \in \NN_0^k$ we have
 \begin{align*}
		\Pr \bc{ \offdist{j}{i} = \bc{a_1, \ldots, a_k}} &:=  \frac{ \Pr \bc{ \degdist_{i}=\bc{a_1,  \ldots, a_{j-1}, a_j + 1, a_{j+1},\ldots, a_k}}}{\Pr\bc{\degdist_{ij} \geq 1 }}. 
		\end{align*}
\end{definition}
Equivalently, $\offdist{j}{i}$ and $\degdist_{i}$ satisfy the following relation.   Let $\cE_{ij}$ be the event $\degdist_{ij}\geq 1$. Then for any $\bc{a_1, \ldots, a_k} \in \NN_0^k$ such that $a_j\geq 1$ we have 
\begin{align*}
	 \Pr \bc{\offdist{j}{i}=\bc{a_1, \ldots, a_{j-1}, a_j -1, a_{j+1}, \ldots, a_k}}= \Pr \bc{\degdist_{i}=\bc{a_1, \ldots, a_k } \big\vert \cE_{ij} }.
\end{align*}

We will talk about \emph{generating vertices with types} according to a distribution $\mathcal{D}$ on $\NN_0^k$,
by which we mean that we generate a vector $\bc{z_1,\ldots,z_k}$ according to $\mathcal{D}$,
and for each $i \in \brk{k}$ we generate $z_i$ vertices of type $i$.
Usually, $\mathcal{D}$ will be $\degdist_i$ or $\offdist{j}{i}$ for some $i,j \in \brk{k}$.
Depending on the context, we may also talk about generating \emph{neighbours}, \emph{children}, \emph{half-edges} etc.\ with types,
in which case the definition is analogous.

\begin{definition} \label{def:GenMulti}
Given $\cD \in \cP\bc{\NN_0^k}$
and a vector $\vec q = \bc{q_1,\ldots,q_k} \in \bc{\cP\bc{\alphabet}}^k$ of probability distributions on $\alphabet$,
let us define a multiset $\MSet{\cD}{\vec{q}}$ of elements of $\alphabet$ as follows.
\begin{itemize}
	\item  Generate a vector  $\bc{a_1, \ldots, a_k}$ according to $\cD$.
	\item  For each $j \in \brk{k}$ independently, select $a_j$ elements of \ $\alphabet$ independently according to $q_{j}$. Call the resulting multiset $\cM_{j}$.	
	\item  Define $\MSet{\cD}{\vec{q}}:= \biguplus_{j=1}^k \cM_{j}.  \footnote{The symbol $\biguplus$ denotes the multiset union of two multisets $A,B$, e.g.\ if $A=\mset{a,a,b}$ and $B=\mset{a,b,c,c}$ then $A \biguplus B = \mset{a,a,a,b,b,c,c}$.}$
\end{itemize}
\end{definition}

The motivation behind this definition is that $\cD$ will represent a distribution of neighbours with types,
typically $\degdist_i$ or $\offdist{j}{i}$ for some $i,j\in \brk{k}$.
Meanwhile $\vec q$ will represent the distributions of messages from the vertices of various types,
typically chosen according to the appropriate entry of a probability distribution matrix,
which are heuristically almost independent.
Thus $\MSet{\cD}{\vec q}$ describes a random multiset of incoming messages at a vertex with the appropriate distribution.

We can now formally describe how the WP update function affects the distribution of messages,
as described by a probability distribution matrix on $\alphabet$.

\begin{definition} \label{def:distribMulti}
Given a probability distribution matrix $Q$ on $\alphabet$ with rows $\Matrdist{Q}{1},\ldots,\Matrdist{Q}{k}$, let $\distf\bc{Q}$ denote the probability distribution matrix $R$ on $\alphabet$ where each entry $\Matrdist{R}{i,j}$
is the probability distribution on $\alphabet$ given by
$$
\Matrdist{R}{i,j}:=\wpf\bc{\MSet{\offdist{j}{i}}{\Matrdist{Q}{i}}}.
$$
Further, let $\distf^t\bc{Q} = \distf\bc{\distf^{t-1}\bc{Q}}$ denote the $t^{th}$ iterated function of $\distf$ evaluated at $Q$.
In order to ease notation, we sometimes denote $\distf^t\bc{Q}$ by $\distphi{Q}{t}$ when $\distf$ is clear from the context.
\end{definition}

In an idealised scenario, this update function precisely describes how the probability distribution matrix should change over time:
along a directed edge of type $\bc{i,j}$, the messages in the next step will be determined by \emph{other} incoming messages at the source vertex;
the neighbours and their types may be generated according to $\offdist{j}{i}$; the corresponding messages are generated according to $\Matrdist{Q}{i}$.

We will ultimately show that this idealised scenario is indeed a reasonable approximation. But we are also interested
in what occurs when we iterate this process from an appropriate starting matrix. Does it converge to some limit? In order to quantify this,
we need the following metric on the space of probability distribution matrices, which is a simple extension of the standard
total variation distance for probability distributions, denoted $\dTV\bc{\cdot,\cdot}$. 

\begin{definition}\label{def:dtv}
The total variation distance of two $k\times k$ probability distribution matrices $Q$ and $R$ on the same set $S$ is defined as
$\dTV\bc{Q,R} := \sum_{i,j\in\brk{k}} \dTV\bc{\Matrdist{Q}{i,j}, \Matrdist{R}{i,j}}$. 
\end{definition}

It is elementary to check that $\dTV$ is indeed a metric on the space of $k\times k$ probability distribution matrices on~$\alphabet$,
and whenever we talk of limits in this space, those limits are with respect to this metric.
We can now define the key notion of a \emph{stable WP limit}, which is fundamental to Theorem~\ref{thm:main}.

\begin{definition}\label{def:FPdistfun} Let $P$ be a probability distribution matrix on $\alphabet$
and $\wpf:\bigcup_{d\geq0}\ms\alphabet d\to\alphabet$ be a WP update rule.
	\begin{enumerate}
		\item We say that $P$ is a \emph{fixed point} if $\distf\bc{P}=P$.
		\item A fixed point $P$ is \emph{stable} if $\distf$ is a contraction on a neighbourhood of $P$ with respect to
		the total variation distance $\dTV$ as defined in Definition~\ref{def:dtv}.
		\item We say that $P$ is the \emph{stable WP limit} of a probability distribution matrix $\initialdistribution$ on $\alphabet$
		if $P$ is a stable fixed point,
		and furthermore the limit $\distf^*\bc{\initialdistribution}:=\lim_{t \to \infty} \distf^t\bc{\initialdistribution}$ exists and equals $P$.
	\end{enumerate} 
\end{definition}

\subsection{Assumptions on the $\wpg$ model}\label{sec:assump}

In order to apply the results of this paper, we will need the
random graph $\GG$ to be reasonably well-behaved; formally,
we require a number of relatively mild properties to be satisfied.
In order to introduce the assumptions, we need to introduce some terminology and notation.

Recall that depending on the application,
the numbers of vertices $n_1,\ldots,n_k$ in each of the $k$ classes may be random, or some may be random and others deterministic. For example,
if we consider the standard bipartite factor graph of a binomial random $r$-uniform hypergraph $H^r\bc{n,p}$,
then one class representing the vertices of $H^r\bc{n,p}$ would have $n_1=n$ vertices deterministically,
while the other class representing the edges of $H^r\bc{n,p}$ would have $n_2\sim \Bin\bc{\binom{n}{r},p}$ vertices.

We seek to model this situation, which we do by introducing a
probability distribution vector $\vprobvec = \bc{\vprobvec_1,\ldots,\vprobvec_k} \in \cP\bc{\NN_0^k}$.
Each $\vprobvec_i$ is a probability distribution on $\NN_0$, although in general they may be dependent on each other.
As mentioned informally earlier, we will also have an implicit parameter $n$, so $\vprobvec=\vprobvec\bc{n}$, and we are interested in asymptotics as $n\to \infty$.
Note that as in the example of factor graphs of hypergraphs above, and in many other examples, we could certainly have $\vprobvec_1=n$ deterministically.
As previously mentioned, we will often conflate random variables and their associated probability distributions; in particular we
will use $n_i$ instead of $\vprobvec_i$.

\begin{definition}\label{def:typedegreesequence} For a $k$-type graph $G$, the type-degree of a vertex $v\in V\bc{G}$, which we denote by $\typedeg\bc{v}$, is the sequence $\bc{i,\typedegind_1, \ldots, \typedegind_k } \in \brk{k}\times \NN_0^k$ where $i$ is the type of $v$ and where 
$\typedegind_j$ is the number of neighbours of $v$ of type $j$. Moreover, the \emph{type-degree sequence} $\vD\bc{G}$ of $G$ is the sequence $ \bc{ \typedeg\bc{v}}_{v \in V\bc{G}}$ of the type-degrees of all the vertices of $G$. 
\end{definition}

This is an obvious generalisation of the standard degree sequence in which we additionally keep track of the types
of the vertices and their neighbours.
We note that for $ \bc{ \typedeg\bc{v}}_{v \in V\bc{G}}$ to be well defined, we need an order for the set of vertices $V\bc{G}$.
Since the order of the type-degree sequence will not play any role in future, we may choose such an order arbitrarily.

We also need to describe the local structure of the graph in terms of a branching process which depends on
the degree distributions $\degdist_1,\ldots,\degdist_k$.

\begin{definition} \label{def:GWproc} Let $\degdist_{1},\ldots,\degdist_k \in \cP\bc{\NN_0^k}$
and for all $\bc{i,j} \in \Admit$, let  $ \offdist{j}{i}$ be as in Definition~\ref{def:offspringdist}.
 For each $i \in \brk{k}$, let $\cT_i:=\cT_i \bc{\degdist_1, \ldots, \degdist_k}$ denote a $k$-type Galton-Waltson process defined as follows:
	\begin{enumerate}
		\item The process starts with a single vertex $u$ of type $i$.
		\item Generate children of $u$ with types according to $\degdist_i$.  
		\item Subsequently, starting from the children of $u$, further vertices are produced recursively according to the following rule: for every  vertex $w$ of type $h$ with a parent $w'$ of type $\ell$, generate children of $w$ with types according to $\offdist{\ell}{h}$ independently.
	\end{enumerate}
	Moreover, for  $r \in \NN_0$ we denote by $\cT_{i}^r$  the branching process $\cT_{i}$  truncated at depth $r$.

\end{definition}

	It will be part of our assumptions on $\wpg$ that the branching processes $\cT_i$ do indeed describe the local structure of $\wpg$ \whp.
To quantify this statement, we will need to compare the distributions of the $\cT_i$
with the empirical local structure of $\wpg$.
 Given a $k$-type graph $G$, a vertex $u \in V\bc{G}$ and $r\in\NN_0$,
 let $\ball_{G}\bc{u,r}$ be the $k$-type subgraph of $G$ induced by the neighbourhood of $u$ up to depth $r$
 (i.e.\ all vertices that can be reached by a path of length at most $r$ from $u$), rooted at the vertex $u$.  
We say that two (vertex-)rooted $k$-type graphs $G$ and $G'$ are \emph{isomorphic},
which we denote by $G \cong G'$, if there exists a graph isomorphism between $G$ and $G'$ which preserves the roots and the types of the vertices.
Let $\SetErootedG_\star$ be the set of isomorphism classes of (vertex-)rooted $k$-type graphs (or more precisely, a set consisting
of one representative from each isomorphism class). We define the following empirical neighbourhood distribution for a given $k$-type graph $G$.

\begin{definition}\label{def:EmpiricalNbdist} Let $G$ be a $k$-type graph with parts $V_1\bc{G}, \ldots, V_k\bc{G}$, let $i \in \brk{k}$ and $r \in \NN_0$.
Then for a graph $H \in \SetErootedG_\star$, we define
	$$
	\empNBdist_{i,r}^{G} \bc{ H } :=\frac{1}{ \abs{V_i\bc{G}}   } \sum_{u \in V_i\bc{G}  } \vec{1}\cbc{\ball_G\bc{u,r} \cong H }.
	$$
\end{definition}
In other words, $\empNBdist_{i,r}^{G} \bc{ H }$ is the proportion of vertices in the class $V_i\bc{G}$ whose $r$-depth
neighbourhood in $G$ is isomorphic to $H$.
When the graph $G$ is clear from the context, we will drop the superscript $G$ in $ \empNBdist_{i,r}^{G}$.  

Note that $\empNBdist_{i,r}^{G}$ defines a probability distribution on the class of rooted $k$-type graphs $H$ of depth at most $r$,
and therefore it can be compared with the truncated branching processes $\cT_{i}^r$, which we will do in
Assumption~\ref{AssumeNew} (specifically~\ref{assumpnew:dtv}).
This assumption lays out the various properties that are required for our proofs.
For parameters $a=a\bc{n}$ and $b=b\bc{n}$, we sometimes use the notation $a\ll b$
as a shorthand for $a=o\bc{b}$, and similarly $a\gg b$ for $b=o\bc{a}$. \newpage

\begin{assumption}\label{AssumeNew}
There exist functions
\begin{equation}\label{eq:choiceDelta0}
1\ll \Delta_0 = \Delta_0\bc{n} \ll n^{1/10}
\end{equation}
and $\zeta = \zeta\bc{x} \xrightarrow{x\to \infty} \infty$
and a probability distribution vector $\degdistset:=\bc{\degdist_1,\dots, \degdist_k} \in \bc{\cP\bc{\NN_0^k}}^k$ 
such that
for all $i \in \brk{k}$ and for all $x \in \RR$, we have
\begin{equation} \label{eq:degdisttailbound}
\Pr\bc{ \norm{\degdist_{i}}_1 > x } \le \exp\bc{-\zeta\bc{x}\cdot x}, 
\end{equation}
and such that the random graph $\wpg$ satisfies the following properties:
\begin{enumerate}[label=\textbf{A\arabic*}]
	\item \label{PP:NConcentration} For all $i \in \brk{k}$ we have $\Erw\bc{n_i}= \Theta\bc{n}$ and $\Var\bc{n_i} = o\bc{n^{8/5}}$.
	\item \label{assump:Eqlikely} For any two simple $k$-type graphs $G$ and $H$ satisfying $\vD\bc{G}=\vD\bc{H}$,
		we have $\Pr\bc{ \wpg = G}= \bc{ 1 +o\bc{1} }\Pr\bc{ \wpg = H}$.  
	\item \label{assump:Delta} \Whp\ $\Delta\bc{\wpg} \le \Delta_0$;
	\item \label{assumpnew:dtv} For any $i \in \brk{k}$ and $r \in \NN_0$ we have
		$$ \dTV\bc{\empNBdist_i^r\bc{\wpg}, \mathcal{T}_{i}^r\bc{\degdistset}} \ll \frac{1}{\Delta_0^2} \quad \mbox{ \whp}$$
\end{enumerate}
\end{assumption}

Note that informally,~\ref{assumpnew:dtv} states that
the local structure of $\wpg$ is asymptotically described by the branching processes $\bc{ \cT_{i}}_{i \in \brk{k}}$ with speed of convergence
faster than $1/\Delta_0^2$.
For most random graph models, it is rather easy to verify that~\eqref{eq:choiceDelta0},~\eqref{eq:degdisttailbound}
and~\ref{PP:NConcentration},~\ref{assump:Eqlikely},~\ref{assump:Delta} hold with the appropriate choice of parameters,
and the main difficulty is to bound the speed of convergence of the local structure as required by~\ref{assumpnew:dtv}.

\subsection{Choosing the parameters}\label{sec:parameters}

Given that the truth of the assumptions is fundamentally dependent on the choice of the parameters $\Delta_0,\zeta,\degdistset$,
for which there may be many possibilities, let us briefly discuss how best to choose them.

\subsection*{The probability distribution vector $\degdistset$}
First observe that given the graph model $\GG$, due to~\ref{assumpnew:dtv} there is only one sensible choice for the probability distribution vector
$\degdistset$,
namely the one which describes the local structure of $\GG$ (in the sense of local weak convergence).
For example,
in the case of the Erd\H{o}s-R\'enyi binomial random graph $G\bc{n,d/n}$ for some constant $d$,
we have $k=1$ would choose $\degdistset = \degdist_1 = \degdist_{11}$
to be the $\Po\bc{d}$ distribution. On the other hand, for the analogous balanced bipartite random graph $G\bc{n,n,d/n}$
we would set 
$\degdistset = \bc{\degdist_1,\degdist_2}$, where $\degdist_1=\bc{\degdist_{11},\degdist_{12}} = \bc{0,\Po\bc{d}}$ and similarly for $\degdist_2$.

\subsection*{The function $\zeta$}
This function only appears in the restriction, given by~\eqref{eq:degdisttailbound}, that the tail bounds of the $\degdist_i$ distributions
decay super-exponentially fast.
As such, we can simply set $\zeta\bc{x} := \min_{i \in \brk{k}}\frac{-\ln \Pr\bc{ \norm{\degdist_{i}}_1 > x }}{x}$ for all $x$.
The assumption demands that this expression tends to infinity.

\subsection*{The degree bound $\Delta_0$}
The most critical property of $\Delta_0$ is~\ref{assump:Delta}, which states that \whp\ it is an upper bound on the maximum degree of $\GG$.
To make the task of proving~\ref{assumpnew:dtv} easier,
it is most convenient to choose $\Delta_0$ as small as possible such that~\ref{assump:Delta} is satisfied.
However, if in fact a bounded $\Delta_0$ would suffice for this purpose (for example when considering
random $d$-regular graphs), we would  choose $\Delta_0$ tending to infinity arbitrarily
slowly in order to ensure that the lower bound in~\eqref{eq:choiceDelta0} is satisfied. In fact, the condition $\Delta_0 \gg 1$
in~\eqref{eq:choiceDelta0}
is imposed purely for technical convenience later on, and (by choosing $\Delta_0$ to grow
arbitrarily slowly if necessary) does not actually impose any additional restrictions on the random model.

A typical non-regular scenario would be that we have $\Theta\bc{n}$ vertices whose degrees are Poisson distributed with
bounded expectation, in which case we could choose $\Delta_0=\ln n$.

\bigskip

Assumption~\ref{AssumeNew} actually contains a further hidden parameter which, for simplicity,
we just chose to be $1/5$. More precisely, we have the following.

\begin{remark}
In Assumption~\ref{AssumeNew}, the conditions
\ref{PP:omegadeltanpower} and~\eqref{eq:choiceDelta0} can be replaced by the assumption that there exists some constant $0<\beta<1/3$
such that:

\begin{tabular}{rl}
\textbf{\eqref{eq:choiceDelta0}'} &  $1 \ll \Delta_0 \ll n^{\beta/2} $; \\
\textbf{\eqref{PP:NConcentration}'} &  For all $i\in\brk{k}$, we have $\Erw\bc{n_i}=\Theta\bc{n}$, and $\mbox{Var} \bc{n_i } = o\bc{n^{2\bc{1-\beta}}}$.
\end{tabular}

\end{remark}
In Assumption~\ref{AssumeNew} we arbitrarily chose $\beta=1/5$ since the only additional restrictions this places on the model~$\wpg$,
once we account for being able to choose other parameters appropriately, are that \whp\ $\Delta\bc{\wpg} \ll n^{1/10}$
and\linebreak[4] $\Var\bc{ \abs{V_i}}= o\bc{n^{8/5}}$. It seems unlikely that there will be a natural model $\wpg$ for which this fails to hold,
but for which it would be true for some different choice of $\beta$. Nevertheless, the proof would still go through in the more general case.

Let us make one further remark regarding~\ref{assump:Eqlikely}, which states that any two graphs with the same type-degree sequence
are asymptotically equally likely under $\GG$.
 This condition is not satisfied for certain natural random graph models,
for example random triangle-free graphs. However, a standard trick allows us to weaken the conditions a little
such that this model would indeed be covered.

\begin{remark}\label{rem:assconditioning}
Assumption~\ref{AssumeNew} can be replaced by the following:\\
There is a random graph model $\GG^*$ and an event $\cE$ such that
\begin{itemize}
\item $\Pr_{\GG^*}\bc{\cE} = \Theta\bc{1}$;
\item $\GG \sim \GG^*|_\cE$, i.e.\ $\GG^*$ conditioned on $\cE$ is precisely $\GG$;
\item $\GG^*$ satisfies Assumption~\ref{AssumeNew}.
\end{itemize}
\end{remark}

So for example when $\GG$ is the random triangle-free graph, we would choose $\GG^*$ to be the unconditioned random graph,
and $\cE$ to be the event that $\GG^*$ is triangle-free. The reason the proof still goes through is that our results can be applied to $\GG^*$
and give a high probability statement, which then also holds \whp\ in the space conditioned on the $\Theta\bc{1}$-probability event $\cE$.
We omit the details.

\subsection{Some simple consequences}

We next collect a few consequences of the assumptions that will be convenient later.
Assumption~\ref{AssumeNew} guarantees the existence of some parameters,
but we will need to fix more for the proof. Specifically, we have the following.

	\begin{proposition}\label{prop:Assume1}
	If Assumption~\ref{AssumeNew} holds, then
	there exists a function $F:[0,\infty) \rightarrow [1,\infty)$ and functions $\omega_0=\omega_0\bc{n}, c_0 = c_0\bc{n},d_0=d_0\bc{n}$ such that:
\begin{enumerate}[label=\textbf{F\arabic*}]
\item \label{PP:F} $F$ is monotonically increasing and invertible;
\item \label{PP:Fseq} For any sequences of real numbers $a=a\bc{n}$ and $b=b\bc{n}$, if $1\le a \ll b$ then $F\bc{a} \ll F\bc{b}$;
\item \label{PP:Fmon} For any sequence of real numbers $a=a\bc{n}\gg 1$ and for any constant $c >0$ we have $F\bc{a} \gg \exp\bc{ ca }$;
\item \label{PP:FandZ} There exists a sufficiently large $x_0\ge 0$ such that for all $x > x_0$ and all $i\in\brk{k}$, we have 
$$ \Pr\bc{ \norm{\degdist_{i}}_1 > x } \leq \frac{1}{F\bc{x}}. $$
\end{enumerate}
	Moreover,
\begin{enumerate}[label=\textbf{P\arabic*}]
	\item \label{PP:omegadeltanpower}\label{PP:omegan} $1 \ll \Delta_0^2 \ll \omega_0 \ll n^{1/5} $;
	\item \label{PP:choiced0} $F^{-1}\bc{\Delta_0^2} \ll d_0 \ll \ln\omega_0$;
	\item \label{PP:choicec0} $\Delta_0 \exp\bc{Cd_0}, \Delta_0^2  \ll c_0 \ll F\bc{d_0},  \omega_0$
		for any constant $C$, 
\end{enumerate}
and the random graph $\GG$ satisfies the following.
	\begin{enumerate}[label=\textbf{B\arabic*}] 	
			\item \label{assump:LocS}  
			 For any $i \in \brk{k}$ and $r \in \NN_0$ we have 			$$ \dTV\bc{\empNBdist_i^r\bc{\wpg}, \mathcal{T}_{i}^r\bc{\degdistset}} \le \frac{1}{\omega_0} \quad \mbox{ \whp}$$ 
	\end{enumerate}
	\end{proposition}

For the rest of the paper, we will fix parameters $\Delta_0,\omega_0,c_0,d_0$ and a function $F$ as in Assumption~\ref{AssumeNew}
and Proposition~\ref{prop:Assume1}.
An obvious consequence of~\eqref{PP:choicec0} is that for any constant $t_0$,
	\begin{equation}\label{eq:secondchoicec0}
	\max\{d_0,\Delta_0\}\cdot \abs{\alphabet}^{2\bc{t_0 +2}d_0}\le \Delta_0 \cdot \abs{\alphabet}^{2\bc{t_0 +3}d_0} = o\bc{c_0},
	\end{equation}
	and this form will often be the most convenient in applications.
Before proving Proposition~\ref{prop:Assume1}, we prove an auxiliary claim which will be helpful both for this proof and later in the paper.

\begin{claim}\label{PP:FDelta} 
If \ref{PP:omegadeltanpower}, \ref{PP:F} and \ref{PP:Fmon} hold, then  $F^{-1}\bc{\Delta_0^2} \ll \ln\omega_0$. 
\end{claim}
\begin{proof}
Suppose it is not true that $F^{-1}\bc{\Delta_0^2} \ll \ln\omega_0$. Then (passing to a subsequence of necessary)
there exists some constant $c>0$ such that $F^{-1} \bc{ \Delta_0^2 } \ge c \ln \bc{ \omega_0}$.
Applying $F$ to both sides, we deduce $\Delta_0^2 \ge F \bc{c \ln \bc{\omega_0}}$, since $F$ is monotonically increasing by~\ref{PP:F}.
Moreover, by \ref{PP:Fmon} we have $F\bc{c \ln \bc{\omega_0}} \gg \omega_0$, so we conclude that $\Delta_0^2 \gg \omega_0$,
which contradicts \ref{PP:omegadeltanpower}. 
\end{proof}

In the proof of Proposition~\ref{prop:Assume1}, for simplicity we will allow functions to take the values $\pm \infty$,
and define expressions involving division by $0$ or $\infty$ 
in the obvious way.
This avoids annoying technical complications required to deal with some special cases---turning
this into a formally correct proof would be an elementary exercise in analysis.

\begin{proof}[Proof of Proposition~\ref{prop:Assume1}]
First let us fix $F_1\bc{x}:= \min_{i \in \brk{k}} \frac{1}{ \Pr\bc{ \norm{\degdist_{i}}_1 > x }}$
and observe that $F_1\bc{x} = \exp\bc{\zeta_1\bc{x}\cdot x}$ for some non-negative function $\zeta_1\bc{x} \xrightarrow{x\to \infty}\infty$.
This means that $F_1$ satisfies conditions~\ref{PP:Fmon} and~\ref{PP:FandZ}, but not necessarily conditions~\ref{PP:F} and~\ref{PP:Fseq}.
We therefore modify this function slightly. More precisely, we can modify the function $\zeta_1$ to obtain $\zeta_2$
satisfying:
\begin{itemize}
\item $\zeta_2\bc{0}=0$;
\item $\zeta_2\bc{x}$ is continuous and monotonically strictly increasing;
\item $\zeta_2\bc{x} \le \zeta_1\bc{x}$ for all sufficiently large $x \in \RR$;
\item $\zeta_2\bc{x} \xrightarrow{x\to \infty} \infty$.
\end{itemize}
We now set
$
F\bc{x}:= 
\exp\bc{\zeta_2\bc{x}\cdot x}.
$
It can be easily checked that $F$ satisfies all the necessary conditions.

Now let us set $\omega_0:= \Delta_0^2 \cdot \omega$, where $\omega=\omega\bc{n}$ is a function
tending to infinity arbitrarily slowly. Since $1\ll \Delta_0^2 \ll n^{1/5}$, if $\omega$ grows sufficiently slowly,
\ref{PP:omegadeltanpower} is also satisfied.
Similarly, since~\ref{assumpnew:dtv} is satisfied,
if $\omega$ grows sufficiently slowly, we also have~\ref{assump:LocS}.

We also set $d_0 := F^{-1}\bc{\Delta_0^2}\cdot \omega$.
Then the lower bound in~\ref{PP:choiced0} is clearly satisfied.
Furthermore Claim~\ref{PP:FDelta} shows that the upper bound also holds provided
$\omega$ tends to infinity slowly enough.

Finally we will show that, provided $\omega$ grows slowly enough,
$\Delta_0\exp\bc{ C d_0 } \ll \Delta_0^2  \ll F\bc{d_0},  \omega_0$,
and then picking $c_0 := \Delta_0^2 \cdot \omega$, we have that~\ref{PP:choicec0} holds.

We first recall that $F\bc{x}=\exp\bc{\zeta_2\bc{x}\cdot x}$, where $\zeta_2\bc{x} \xrightarrow{x\to \infty} \infty$.
Thus $F^{-1}\bc{x} = \frac{\ln x}{\zeta_3\bc{x}}$, where $\zeta_3\bc{x} = \zeta_2\bc{F^{-1}\bc{x}} \xrightarrow{x\to \infty} \infty$.
It follows that, for any constant $C>0$, we have $\exp\bc{Cd_0} = \exp\bc{\frac{C\bc{\ln \Delta_0}\omega}{\zeta_3\bc{\Delta_0^2}}} \le \exp\bc{\frac{\bc{\ln \Delta_0}\omega}{\zeta_4\bc{n}}}$
for sufficiently large~$n$ and
for some appropriate function $\zeta_4\bc{n}\xrightarrow{n\to \infty} \infty$ (which is independent of $C$).  
By choosing $\omega \ll \zeta_4$, we have
$\exp\bc{Cd_0} \ll \Delta_0$ and therefore also $\Delta_0\exp\bc{Cd_0} \ll \Delta_0^2$.
Now to complete the proof, observe that $d_0 \gg F^{-1}\bc{\Delta_0^2}$ by definition, and therefore~\ref{PP:Fseq} implies that
$\Delta_0^2 \ll F\bc{d_0}$.
On the other hand, $\Delta_0^2 \ll \omega_0$ by definition of $\omega_0$.
\end{proof}

A further consequence of the assumptions is that the degree distributions have bounded moments.

\begin{remark}\label{P7}\label{rem:finitemoments}
	Claim \ref{PP:FDelta} and~\ref{PP:FandZ} together imply that for all $i \in \brk{k}$,
	the distribution $\|\degdist_{i}\|_1$ of the total degree of a vertex of type $i$ has
	finite moments, i.e.\ $\Erw\bc{\|\degdist_{i}\|_1^s}$ is finite for any $s \in \NN$,
	and in particular for any $i,j\in \brk{k}$ and $s \in \NN$ the moment $\Erw\bc{\degdist_{ij}^s}$ are finite. It also follows that for every admissible pair $\bc{i,j} \in \Admit$, the moments $\Erw\bc{ \norm{\offdist{j}{i}}_1^s}$ are finite
		(this can be verified with an elementary check). We will often use these facts during the proofs.
\end{remark}

We will also need the simple observation that the class sizes are reasonably concentrated around their expectations.
\begin{claim}\label{claim:classsizes}
\Whp\ for all $i \in \brk{k}$ we have
$
n_i = \bc{1+o\bc{\frac{1}{\omega_0}}}\Erw\bc{n_i}.
$
\end{claim}
\begin{proof}
By \ref{PP:NConcentration}, for all $i \in \brk{k}$, we have  $\Erw\bc{n_i} = \Theta\bc{n}$ and $\Var\bc{ n_i } = o\bc{n^{8/5}}$.
Let $\omega = \omega\bc{n} := \frac{n^{8/5}}{\max_{i \in \brk{k}}\Var\bc{n_i}}$, so in particular $\omega \to \infty$.
(Note that if $\Var\bc{n_i}=0$ for all $i$, then the claim is trivial, so we may assume that $\omega$ is well-defined.)
Then
Chebyshev's inequality implies that
$$
\Pr\bc{\big|n_i-\Erw\bc{n_i}\big|\ge n^{4/5}}
\le \Pr\bc{\big|n_i-\Erw\bc{n_i}\big|\ge \sqrt{\omega\cdot \Var\bc{n_i}}} \le \frac{1}{\omega} = o\bc{1}.
$$
In other words, \whp\ 
$
n_i= \bc{1+ O\bc{\frac{1}{n^{1/5}}}}\Erw\bc{n_i}
$,
and
since $\omega_0\ll n^{1/5}$ by~\ref{PP:omegadeltanpower}, taking a union bound over all $i \in \brk{k}$ gives the desired result.
\end{proof}

\section{An alternative model}\label{sec:apg}

Although our main result is primarily a statement about $\wpg$,
a key method in this paper is to switch focus away from this model to a second model, denoted $\apg$,
which is easier to analyse. To introduce this second model, we need some more definitions.

\subsection{Message histories}
Let $\messgraphs{n}$ denote the set of \emph{$\alphabet$-messaged graphs} on vertex set $[n]$,
i.e.\ graphs on $[n]$ in which each edge $uv$ comes equipped with directed messages
$\mu_{u\to v},\mu_{v\to u} \in \alphabet$.

We will denote by $\mu_{u\to v}\bc{t}$ the message from $u$ to $v$
after $t$ iterations of $\WP$, and refer to this as the \emph{\messaget{t}} from $u$ to $v$.
Alternatively, we refer to the \emph{\inmessaget{t}} at $v$ or the \emph{\outmessaget{t}} at $u$
(this terminology will be especially helpful later when considering half-edges).
In all cases, we may drop $t$ from the notation if it is clear from the context.

In fact, we will need to keep track not just of the current Warning Propagation messages along each edge,
but of the entire history of messages.
For two adjacent vertices $u,v$, define
the \emph{\historyt{t}\ from $u$ to $v$} to be the vector
$$
\messagehistory{u}{v}{t} := \bc{\mu_{u\to v}\bc{0},\ldots ,\mu_{u\to v}\bc{t}} \in \alphabet^{t+1}.
$$
We will also refer to $\messagehistory{u}{v}{t}$ as the \emph{\instoryt{t}} at $v$,
and as the $\emph{\outstoryt{t}}$ at $u$. The \emph{\storyt{t}} at $v$
consists of the pair $\bc{\messagehistory{u}{v}{t},\messagehistory{v}{u}{t}}$, i.e.\ the \instoryt{t}
followed by the \outstoryt{t}. It will sometimes be more convenient to consider the sequence consisting
of the \instoryt{t} followed by just the \outmessaget{0}, which we call the \emph{\inpt{t}}.
In all cases, we may drop $t$ from the notation
if it is clear from the context.

We denote by $\histgraphs{n}{t}$ the set of $\alphabet^{t+1}$-messaged graphs on vertex set $[n]$ --
the labels along each directed edge, which come from $\alphabet^{t+1}$, will be the
\historiest{t}. \footnote{Note that the definition of $\histgraphs{n}{t}$ makes no assumption that the
\histories\ along directed edges arise from running Warning Propagation -- in principle, they
could be entirely inconsistent -- although of course in our applications, this will indeed be the case.}

With a slight abuse of notation, for $t_1< t_2$ we will identify two graphs
$G \in \histgraphs{n}{t_1}$ and $H \in \histgraphs{n}{t_2}$, whose messages are given by $\mu^{\bc{G}}$ and $\mu^{\bc{H}}$ respectively,
if
\begin{itemize}
\item $E\bc{G}=E\bc{H}$;
\item $\mu^{\bc{G}}_{u\to v}\bc{t} = \mu^{\bc{H}}_{u\to v}\bc{t}$ for all $t \le t_1$;
\item $\mu^{\bc{H}}_{u\to v}\bc{t} = \mu^{\bc{H}}_{u\to v}\bc{t_1}$ for all $t_1 < t \le t_2$.
\end{itemize}
In other words, the underlying graphs are identical, the \historiest{t_1} are identical, and subsequently no messages change in $H$.
In particular, this allows us to talk of \emph{limits} of messaged graphs $G_t \in \histgraphs{n}{t}$ as $t \to \infty$.

\begin{definition}\label{def:wpg}
For any $t \in \NN$ and probability distribution matrix $\initialdistribution$ on $\alphabet$,
let $\wpg_{t} = \wpg_{t}\bc{n,\initialdistribution} \in \histgraphs{n}{t}$ be the random $\alphabet^{t+1}$-messaged graph produced as follows.
\begin{enumerate}
\item Generate the random graph $\wpg$.
\item   Initialise each message $\mu_{u \to v}\bc{0}$ for each directed edge $\bc{u,v}$ independently at random according to
$\initialdistribution[i,j]$ where $i$ and $j$ are the types of $u$ and $v$ respectively.
\item Run Warning Propagation for $t$ rounds according to update rule $\wpf$.
\item Label each directed edge $\bc{u,v}$ with the \story\ $\bc{\mu_{u\to v}\bc{0},\ldots,\mu_{u\to v}\bc{t}}$ up to time $t$.
\end{enumerate}
We also define $\wpg_* := \lim_{t\to \infty}\wpg_t$, if this limit exists.
\end{definition}

We aim to move away from looking at $\wpg_t$
and instead to consider a random graph model $\apg_t$ in which we first generate half-edges at every vertex,
complete with \stories\ in both directions, and only subsequently reveal which
half-edges are joined to each other; thus we construct a graph in which the
WP messages are known a priori. The trick is to do this in such a way that the
resulting random messaged graph looks similar to $\wpg_t$.

In order to define this random model, we need a way of generating a \history\ randomly, but accounting for
the fact that the entries of a \history\ are, in general, heavily dependent on each other, which we do in
Definition~\ref{def:GWmsghistory}. We first need to define a variant of the $\cT_i$ branching trees.

An \emph{edge-rooted graph} is a simple graph with a distinguished directed edge designated as root edge.
When we have an edge-rooted \emph{tree} rooted at the directed edge $\bc{u,v}$, we will think of $v$ as the parent of $u$, and in all such
situations
$v$ will have no other children. More generally, whenever we talk of messages along an edge of such a  tree,
we mean along the directed edge from child to parent.

	We will also need to describe the
	part of the local structure that influences a message along a directed edge $\bc{u,v}$. This motivates the following definition.

\begin{definition}\label{def:GWprocedge}
	Let $\degdist_{1},\ldots,\degdist_k$ be probability distributions on $\NN_0^k$ and for all $i,j \in \brk{k}$, let  $\offdist{j}{i}$ be as in
	Definition~\ref{def:offspringdist}. For each $\bc{i,j} \in \Admit$, let $\cT_{ij}:=\cT_{ij} \bc{\degdist_1, \ldots, \degdist_k}$ denote a $k$-type Galton-Waltson process defined as follows:
	\begin{enumerate}
		\item The process starts with a directed root edge $\bc{u,v}$ where $u$ has type $i$ and $v$ has type $j$.
		We refer to $v$ as the parent of $u$, and $v$ will have no further children.
		\item Subsequently, starting at $u$, vertices are produced recursively according to the following rule: for every vertex $w$ of type $h$ with a parent $w'$ of type $\ell$, generate children of $w$ with types according to $\offdist{\ell}{h}$ independently.
	\end{enumerate}
Moreover, for  $r \in \NN_0$ we denote by  $\cT_{ij}^r$ the branching  $\cT_{ij}$ truncated at depth $r$.
\end{definition}

	Note that the process $\cT_{ij}$ can equivalently be produced by taking the process $\cT_i$ conditioned on the root $u$
having at least one child $v$ of type $j$, deleting the entire subtree induced by the descendants of $v$ and rooting the resulting tree at the directed edge $\bc{u,v}$.

\begin{definition} \label{def:GWmsghistory}
Given a probability distribution matrix $Q$ on $\alphabet$, for each $i,j \in \brk{k}$ we define random variables
$X_{ij}^{\bc{0}}, X_{ij}^{\bc{1}}, X_{ij}^{\bc{2}}, \ldots$ as follows.
Let $T_{ij}$ be a randomly generated instance of the process $\mathcal{T}_{ij}$ defined in Definition~\ref{def:GWprocedge}.
	\begin{enumerate}
		\item Initialise all messages in $T_{ij}$ according to $Q$. 
		\item For each $t \in \NN_0$, let $X_{ij}^{\bc{t}}:=\mu_{u \rightarrow v}\bc{t}$ be the message from $u$ to $v$ after $t$ iterations of Warning Propagation according to the update rule $\wpf$ where $v$ is the root of $T_{ij}$ and $u$ its only child.
	\end{enumerate}

Finally, for each $t \in \NN_0$, let $\vdistf{t}\bc{Q}$ be the probability distribution matrix $R$ on $\alphabet^{t+1}$
where each entry $\Matrdist{R}{i,j}$ is the distribution of \ $\bc{ X_{ij}^{\bc{0}}, \ldots, X_{ij}^{\bc{t}}}$. 
As in Definition~\ref{def:distribMulti}, in order to ease notation, we sometimes denote $\vdistf{t}\bc{Q}$ by $\vdistphi{Q}{t}$. 
\end{definition}
Note that $\vdistphi{Q}{t}$ is \emph{not} a vector $\bc{Q^{\bc{0}}, \ldots, Q^{\bc{t}}}$ of probability distribution matrices,
but is instead a matrix in which every entry is a probability distribution on vectors of length $t+1$.

Note also
 that while it is intuitively natural to expect that the marginal distribution of $\Matrdist{\vdistphi{Q}{t}}{i,j}$ on the $\ell$-th entry has the distribution of $\Matrdist{\distphi{Q}{\ell}}{i,j}$,
which motivates the similarity of the notation, this fact is not completely trivial.
We will therefore formally prove this in Claim~\ref{claim:localhistorydistribution}.

\subsection{The random construction}
We define the \emph{\incompilationt{t}} at a vertex $v$ to be the multiset of \inpst{t} at $v$,
and the \emph{\incompseqt{t}} is the sequence of \incompilationst{t} over all vertices of $[n]$.
As before, we often drop the parameter $t$ from the terminology when it is clear from the context.

We can now define the alternative random graph model to which we will switch our focus.

\begin{definition}\label{def:apg}

Given a probability distribution matrix $\initialdistribution$ on $\alphabet$, a sequence $\degdist=\bc{ \degdist_{1}, \ldots, \degdist_k}$
of probability distributions on $\NN_0^k$,
a probability distribution vector
$\vprobvec = \vprobvec\bc{n} \in \cP\bc{\NN_0^k}$ and an integer $t_0$, we construct a random messaged graph
$\apg_{t_0} = \apg_{t_0}\bc{n, \vprobvec, \degdist,\initialdistribution}$
by applying the following steps.

\begin{enumerate}
\item \label{def:apgTypes} Generate $n_1,\ldots,n_k$ according to the probability distribution vector $\vprobvec$, and for each $i \in \brk{k}$
generate a vertex set $V_i$ with $\abs{V_i}=n_i$.
\item \label{def:apgHalf}  For each $i \in \brk{k}$ and for each vertex $v$ in $V_i$ independently, generate an \incompilation\ by: 
\begin{enumerate}
	\item Generating half edges with types $ \bc{i,j}$ for each $j\in \brk{k}$ according to $\degdist_i$;
	\item Giving each half-edge of type $\bc{i,j}$ a \instoryt{t_0}  
	according to $\Matrdist{\vdistphi{\initialdistribution}{t_0}}{j,i}$ independently;
	\item  Giving each half-edge of type $\bc{i,j}$ a \outmessaget{0} according to $\Matrdist{\initialdistribution}{i,j}$ independently of each other and of the \instories.
\end{enumerate}
\item \label{def:apgInstory}Generate \outmessagest{t} for each time $1\le t \le t_0$
according to the rules of Warning Propagation based on the \inmessagest{(t-1)},
i.e.\ if the \instoriest{t_0} at $v$,
from dummy neighbours $u_1,\ldots,u_j$, are $\messagehistory{u_i}{v}{t_0}$,
we set
$$\mu_{v \to u_i}\bc{t} = \wpf\left(\mset{\mu_{u_1 \to v}\bc{t-1},\ldots,\mu_{u_{i-1} \to v}\bc{t-1},\mu_{u_{i+1} \to v}\bc{t-1},\ldots,\mu_{u_j \to v}\bc{t-1}}\right).$$
\item \label{def:apgMatch}  Consider the set of matchings of the half-edges which
are maximum subject to the following conditions:
\begin{itemize}
\item Consistency: a half-edge with \instory\ $\vec \mu_{1} \in \alphabet^{t_0+1}$ and \outstory\ $\vec \mu_{2} \in \alphabet^{t_0+1}$ 
is matched to a half-edge with \instory\ $\vec \mu_{2}$ and \outstory\ $\vec \mu_{1}$;
\item Simplicity: the resulting graph (ignoring unmatched half-edges) is simple.
\end{itemize}
Select a matching uniformly at random from this set and delete the remaining unmatched half-edges.
\end{enumerate}
\end{definition}

From now on we will always implicitly assume that the choice of various parameters is the natural one to compare
$\apg_{t_0}$ with $\wpg_{t_0}$, i.e.\ that $\vprobvec$ is precisely the distribution of the class sizes of $\wpg$
and $\degdist$ is the probability distribution vector which describes the local structure of $\wpg$ as required in Assumption~\ref{AssumeNew},
while $\initialdistribution$ will be the probability distribution matrix according to which we initialise messages in $\wpg$.

We will show later (Claim~\ref{claim:outstorydist}) that the distribution of an \outstory\ is
identical to the distribution of an \instory, which means that the expected number of half-edges
with \story\ $\bc{\vec \mu_1,\vec \mu_2}$ is (almost) identical to the expected number of half-edges
with the dual \story\ $\bc{\vec \mu_2,\vec \mu_1}$. Heuristically, this suggests that almost
all half-edges can be matched up and therefore few will be deleted in Step~\ref{def:apgMatch}.
This will be proved formally in Proposition~\ref{prop:unmatched}.

\begin{remark}\label{rem:outdetermined}
Note that Step~\ref{def:apgInstory} of the construction is an entirely deterministic one --
the \outmessagest{t} at time $t\ge 1$ are fixed by the incoming messages
at earlier times. Therefore all \instories\ and \outstories\ (before the deletion of half-edges)
are in fact determined by the outcome of the random construction in Steps~\ref{def:apgTypes} and~\ref{def:apgHalf}.
\end{remark}

\subsection{Contiguity}

Observe that $\apg_{t_0}$ and $\wpg_{t_0}$ both define random variables in $\histgraphs{n}{t_0}$.
With a slight abuse of notation, we also use $\apg_{t_0}$ and $\wpg_{t_0}$ to denote
the \emph{distribution} of the respective random variables.
Given a $\alphabet^{t+1}$-messaged graph $G \in \histgraphs{n}{t}$,
we will denote by $\proj{G}$ the $\alphabet$-messaged graph in $\messgraphs{n}$
obtained by removing all messages from each history except for the message at time $t$,
i.e.\ the ``current'' message.

There are two main steps in the proof of Theorem~\ref{thm:main}:
\begin{enumerate}
\item Show that $\apg_{t}$ and $\wpg_{t}$ have similar distributions for any constant $t\in \NN$
(Lemma~\ref{lem:contiguity}).
\item Use this approximation to show that, for some large constant $t_0\in \NN$,
the messaged graphs $\proj{\wpg}_{t_0}$ and
$\proj{\wpg}_{*}$ are also very similar, i.e.\ very few further changes are made after
$t_0$ steps of Warning Propagation.
\end{enumerate}

In particular, we must certainly choose $t_0$ to be large enough that $\distf^{t_0}\bc{\initialdistribution}$ is very close to the stable WP limit $P$ of $\initialdistribution$.
It will follow that the distribution of a message
along a randomly chosen directed edge in $\proj{\apg}_{t_0}$ (and therefore also in $\proj{\wpg}_{t_0}$) of type $\bc{i,j}$ is approximately $\Matrdist{P}{i,j}$
(see Claim~\ref{claim:localhistorydistribution}).

We need a way of quantifying how ``close'' two messaged graphs are to each other.
Given sets $A$ and $B$, we use $A\Delta B := \bc{A\setminus B} \cup \bc{B \setminus A}$ to denote
the symmetric difference.

\begin{definition}\label{def:graphsclose}
Given $t\in \NN_0$, two $\alphabet^{t+1}$-messaged graphs $G_1,G_2 \in \histgraphs{n}{t}$ and $\delta >0$, we say that $G_1 \close{\delta} G_2$ if:
\begin{enumerate}
\item $E\bc{G_1}\Delta E\bc{G_2} \le \delta n$;
\item The messages on $E\bc{G_1}\cap E\bc{G_2}$ in the two graphs agree except on a set of size at most $\delta n$.
\end{enumerate}
We further say that $G_1 \vclose{\delta} G_2$ if in fact
the underlying graphs are identical (i.e.\ $E\bc{G_1}\Delta E\bc{G_2}=\emptyset$).
\end{definition}

The crucial lemma that justifies our definition of the $\apg$ model is the following.

\begin{lemma}\label{lem:contiguity}
For any integer $t_0 \in \NN$ and real number $\delta >0$, the random $\alphabet^{t_0+1}$-messaged graphs $\apg_{t_0},\wpg_{t_0}$ can be coupled
in such a way that \whp\ $\apg_{t_0} \close{\delta} \wpg_{t_0}$.
\end{lemma}

This lemma is proved in Section~\ref{sec:contiguity}.

\subsection{Message Terminology}

We have introduced several pieces of terminology related to messages in the graph, which we recall and collect here for easy reference.
For a fixed time parameter $t\in \NN$ and a directed edge, the \emph{\historyt{t}} is the sequence
of messages at times $0,1,\ldots,t$ along this directed edge.
Further, for a (half-)edge or set of (half-)edges incident to a specified vertex,
we have the following terminology.
\begin{itemize}
\item The \emph{\inmessaget{t}} is the incoming message at time $t$.
\item The \emph{\outmessaget{t}} is the outgoing message at time $t$.
\item The \emph{\instoryt{t}} is the sequence of \inmessagest{t'} for $t'=0,\ldots,t$.
\item The \emph{\outstoryt{t}} is the sequence of \outmessagest{t'} at times $t'=0,\ldots,t$.
\item The \emph{\storyt{t}} is the ordered pair consisting of the \instoryt{t} and \outstoryt{t}.
\item The \emph{\inpt{t}} is the ordered pair consisting of the \instoryt{t} and \outmessaget{0}.
\item The \emph{\incompilationt{t}} is the multiset of \inpst{t} over all half-edges at a vertex.
\item The \emph{\incompseqt{t}} is the sequence of \incompilationst{t} over all vertices.
\end{itemize}

When the parameter $t$ is clear from the context, we often drop it from the terminology.

\section{Preliminary results}\label{sec:prelim}

We begin with some fairly simple observations which help to motivate some of the definitions made
so far, or to justify why they are reasonable.
The first such observation provides a slightly simpler way of describing the individual ``entries'', i.e.\ the marginal
distributions, of the probability distribution $\vdistf{t}\bc{\initialdistribution}[i,j]\in \cP\bc{\alphabet^{t+1}}$.

	\begin{claim}\label{claim:localhistorydistribution}
		For any $t',t \in \NN_0$ with $t' \le t$ and for any $i,j\in\brk{k}$, the marginal distribution of $\Matrdist{\vdistf{t}\bc{\initialdistribution}}{i,j}$ on the ${t'}$-th entry 
		is precisely $\Matrdist{\distf^{t'}\bc{\initialdistribution}}{i,j}$, i.e.\ for any $\mu\in \alphabet$ we have
		$$
			\Pr\bc{\Matrdist{\bc{\Matrdist{\vdistf{t}\bc{\initialdistribution}}{i,j}}}{t'}=\mu} =
		\left(\sum_{\substack{\vom=\bc{\mu_0,\ldots,\mu_t} \in \alphabet^{t+1}\\ \mu_{t'}=\mu}} \Pr\bc{\Matrdist{\vdistf{t}\bc{\initialdistribution}}{i,j} = \vom}\right) = \Pr\bc{\Matrdist{\distf^{t'}\bc{\initialdistribution}}{i,j}=\mu}.
		$$
\end{claim}

\begin{proof}
	Using the notation from Definition~\ref{def:GWmsghistory}, we have 
		\begin{align*}	\sum_{\substack{\vom=\bc{\mu_0,\ldots,\mu_t} \in \alphabet^{t+1}\\ \mu_{t'}=\mu}} \Pr\bc{\Matrdist{\vdistf{t}\bc{\initialdistribution}}{i,j} = \vom}  &= 	\sum_{\substack{\vom=\bc{\mu_0,\ldots,\mu_t} \in \alphabet^{t+1}\\ \mu_{t'}=\mu}}\Pr\bc{X_{ij}^{\bc{0}}=\mu_0, \ldots, X_{ij}^{\bc{t}}=\mu_t}
		= \Pr\bc{X_{ij}^{\bc{t'}}=\mu}.\\
		\end{align*}
		We will prove by induction that $\Pr\bc{X_{ij}^{\bc{t'}}= \mu} = \Pr\bc{\Matrdist{\distf^{t'}\bc{\initialdistribution}}{i,j} = \mu}$.
		For $t'=0$, again using Definition~\ref{def:GWmsghistory} the distribution of $X_{ij}^{\bc{0}}$ is simply $\Matrdist{\initialdistribution}{i,j}$,
		so suppose that $t'\ge 1$, that the result holds for $0,\ldots,t'-1$ and for any pair $\bc{h,\ell} \in \brk{k}^2$.
		Let $x_1,\ldots , x_d$ be the children of the root node $u$
		in the $\cT_{ij}$ branching tree defined in Definition~\ref{def:GWprocedge} so the numbers and types of the children
		are given by the distribution $\offdist{j}{i}$. By the recursive nature of the $\cT_{ij}$ branching tree and the induction hypothesis,
		the message from any $x_m$ of type $h$ to $u$ at time $t'-1$ has distribution $\Matrdist{\distf^{t'-1}\bc{\initialdistribution}}{h,i}$
		and this is independent for all vertices. Thus, in order to get the message from $u$ to $v$ at time $t'$,
		we generate a multiset of messages $\MSet{\offdist{j}{i}}{\distf^{t'-1}\bc{\initialdistribution}[i]}$
		as in Definition~\ref{def:GenMulti} and apply the Warning Propagation rule $\wpf$.
		By Definition~\ref{def:distribMulti},
		the distribution of $\wpf \bc{ \MSet{\offdist{j}{i}}{\Matrdist{\distf^{t'-1}\bc{\initialdistribution}}{i}} } $
		is
		$\Matrdist{\distf\bc{\distf^{t'-1}\bc{\initialdistribution}}}{i,j}=\Matrdist{\distf^{t'}\bc{\initialdistribution}}{i,j}$.
\end{proof}

\begin{claim}\label{claim:outstorydist} 
	Given a half-edge of type $\bc{i,j}$ at a vertex  $u$ of type $i$ in the graph $\apg_{t_0}$
	\emph{before} any half-edges are deleted,
	the distribution of its \outstory\ is given by $\Matrdist{\vdistf{t_0}\bc{\initialdistribution}}{i,j}$ .
\end{claim}

We note also that \emph{after} half-edges are deleted, this distribution will remain
asymptotically the same, since \whp\ only $o\bc{n}$ half-edges will be deleted
(see Proposition~\ref{prop:unmatched}).

\begin{proof}
Given such a half-edge at $u$, let us add a dummy vertex $v$ of type $j$ to model the corresponding neighbour of~$u$.
Apart from $\bc{u,v}$, the vertex $u$ has some number $d$ of half-edges with types connected  to dummy vertices $c_1, ..., c_d$ generated according to $ \offdist{j}{i}$.
For each $d' \in [d]$, let $r_{d'}$ be the type of the vertex $c_{d'}$.
Each half-edge  $\bc{c_{d'}, u}$ receives  \instoryt{t_0} according to $\Matrdist{\vdistf{t_0}\bc{\initialdistribution}}{ r_d',i}$.
This is equivalent to endowing each $c_{d'}$ with a $\mathcal{T}_{ r_{d'} i}$  tree independently  where the root edge  is $\bc{c_{d'},u}$, initialising the messages from children to parents in these trees according to $ \initialdistribution$
and running $t_0$ rounds of Warning Propagation.
Combining all these (now unrooted) trees with the additional root edge $\bc{u, v}$,
whose message is also initialised according to $\initialdistribution$ independently of all other messages,
we have a $\mathcal{T}_{ij}$ tree in which all messages are initialised independently
according to $\initialdistribution$.
Then by Definition~\ref{def:GWmsghistory}, $\messagehistory{u}{v}{t_0}$ is distributed as $\Matrdist{\vdistf{t_0}\bc{\initialdistribution}}{i,j}$.   
\end{proof}

Recall that for each $\mu\in\alphabet$, its source and target types are encoded in it.  We define a function to denote these types.
\begin{definition}\label{def:embeddedmsgtype} For a message $\mu\in\alphabet$ with source type $i$ and target type $j$, we define 
	\begin{equation}\label{eq:sourcetarget}
	g\bc{\mu} = \bc{i,j}, \qquad g_1\bc{\mu} = i, \qquad g_2\bc{\mu} = j, \qquad \switch{g}\bc{\mu}=\bc{j,i}.
	\end{equation}

\end{definition}

Recall that not all messages can appear along any edge, and for the same reason not all vectors
of messages are possible as message histories, which motivates the following definition.

\begin{definition}\label{def:consistent_compatible}
	We say that a vector $\vom=\bc{\mu_0, \mu_1, \ldots, \mu_{t}}\in\alphabet^{t+1}$ is \emph{consistent} if the $g\bc{\mu_{t'}}$ are all equal for all $0\le t' \le t$,
	in other words, the source types of the $\mu_{t'}$ are equal and the target types of the $\mu_{t'}$  are equal.
	Let $\Consist{t}\subseteq\alphabet^{t+1}$ be the set of consistent vectors in $\alphabet^{t+1}$.
	For $\vom\in\Consist{t}$ we slightly abuse the notation and define
	$$g\bc{\vom} = g\bc{\mu_0}, \qquad g_1\bc{\vom}=g_1\bc{\mu_0}, \qquad g_2\bc{\vom}=g_2\bc{\mu_0}, \qquad \switch g\bc{\vom} = \switch g\bc{\mu_0}.$$
	Furthermore, we say that $\voone, \votwo\in\Consist{t}$ are \emph{compatible} if
	$g\bc{\voone} = \switch g\bc{\votwo}$, i.e.\ the source type of~$\voone$ is the target type of~$\votwo$ and vice versa.
	Let $\Compat{t}\subseteq \Consist{t}^2$ be the set of directed pairs of compatible vectors.  
\end{definition}

Note that even with this definition, not all consistent vectors are necessarily possible as message histories,
since for example there may be some monotonicity conditions which the vector fails to satisfy.

\begin{definition} Let $Q$ be a probability distribution matrix on $\alphabet$, let $\sigma \in \alphabet$ and $\vec{\mu} \in \Consist{t}$ for some $t \in \NN$. We define $$ \Pr_{\distphi{Q}{t}}\bc{\sigma }:=\Pr\bc{\Matrdist{\distphi{Q}{t}}{g\bc{\sigma}}=\sigma} \mbox{ and  } \, \Pr_{\vdistphi{Q}{t}}\bc{\vec{\mu} }:=\Pr\bc{\Matrdist{\vdistphi{Q}{t}}{g\bc{\vec{\mu}}}=\vec{\mu}}.$$
\end{definition}
In other words, $\Pr_{\distphi{Q}{t}}\bc{\sigma }$ and $\Pr_{\vdistphi{Q}{t}}\bc{\vec{\mu}}$ are the probabilities of obtaining $\sigma$ and $\vec{\mu}$
if we sample from $\distphi{Q}{t}$ and $\vdistphi{Q}{t}$ in the appropriate entry $g\bc{\sigma}$ and $g\bc{\vec{\mu}}$  of those matrices respectively,
the only entries which could conceivably give a non-zero probability.
 
Given an integer $t$ and $\voone,\votwo \in \alphabet^{t+1}$, let
$m_{\voone,\votwo}$ denote the number of half-edges in $\apg_{t}$
with \story\ $\bc{\voone,\votwo}$, i.e.\ with \instory\ $\voone$ and \outstory\ $\votwo$, after Step~\ref{def:apgInstory} of the random construction
(in particular \emph{before} unmatched half-edges are deleted).
Observe that at a single half-edge of type $\bc{i,j}:=\bc{g_1\bc{\voone}, g_2\bc{\voone}}$,  the \instory\ is distributed as
$\Matrdist{\initialdistribution^{\bc{\le t}}}{j,i}$ and by Claim~\ref{claim:outstorydist} the \outstory\ is distributed as  $\Matrdist{\initialdistribution^{\bc{\le t}}}{i,j}$. Moreover, the \instory \ and \outstory \ are independent of each other.
Therefore the probability that the half-edge has \instory\ $\voone$ and \outstory\ $\votwo$ is precisely 
	$$
	q_{\voone,\votwo}:= \begin{cases}
	\Pr_{\initialdistribution^{\bc{\le t}}}\bc{\voone} \cdot \Pr_{\initialdistribution^{\bc{\le t}}}\bc{\votwo} & \mbox{ if } \bc{\voone, \votwo} \in \Compat{t},\\
	0 & \mbox{otherwise.}
	\end{cases}
	$$
The following fact follows directly from the definition of $q_{\voone,\votwo}$.
\begin{fact}
		For any $\bc{\voone, \votwo} \in\alphabet^{t+1}$ we have $q_{\voone, \votwo} =q_{\votwo, \voone}$.
	\end{fact}

We will also define  
\begin{equation}\label{eq:mbardef}
\barm_{\voone, \votwo}:=\begin{cases}
	 \Erw\bc{ \degdist_{g\bc{\voone}}} \Erw\bc{n_{g_1\bc{\voone}}} q_{ \voone, \votwo} & \mbox{ if } \bc{\voone, \votwo} \in \Compat{t}, \\
	 0 & \mbox{ otherwise.}
	\end{cases}
\end{equation}

\begin{claim} \label{claim:ZijandNi}   
		For any $i,j \in \brk{k}$, 
		we have $\Erw\bc{\degdist_{ij}} \Erw\bc{n_{i}}  =
		\bc{1+O\bc{\frac{\Delta_0}{\omega_0}}}\Erw\bc{ \degdist_{ji}} \Erw\bc{n_{j}}$.
		In particular,
		$$
		\barm_{\voone, \votwo} = \bc{1+O\bc{\frac{\Delta_0}{\omega_0}}}\barm_{\votwo, \voone}.
		$$
\end{claim}

\begin{proof}
Let us fix $i,j \in \brk{k}$.
The statement is trivial if $i=j$, and therefore we may assume that this is not the case.
Let us consider the number of edges of $e_{i,j},e_{j,i}$ of types $\bc{i,j}$ and $\bc{j,i}$ respectively in $\wpg$, which must of course be identical.
This can be expressed as $\sum_{v \in V_i} d_{\wpg,j}\bc{v}$, where $d_{\wpg,j}\bc{v}$ denotes the number of neighbours of
$v$ which have type $j$.

Now for each $d \in \NN$,
define $\cS_d$ to be the family of (vertex-)rooted $k$-type graphs of depth $1$ rooted
at a vertex of type $i$, and with exactly $d$ vertices of type $j$.
Then we have
$$
e_{i,j} = \sum_{v \in V_i} d_{\wpg,j}\bc{v} = \sum_{v \in V_i} \sum_{d \in \NN} d\cdot \vec{1}\cbc{d_{\wpg,j}\bc{v}=d}
= \sum_{v \in V_i} \sum_{d \in \NN} \sum_{H \in \cS_d} d \cdot \vec{1}\cbc{B_\wpg\bc{v,1} \cong H}.
$$
Now conditioning on the high probability event that $n_i=\bc{1+o\bc{\frac{1}{\omega_0}}}\Erw\bc{n_i}$ (see Claim~\ref{claim:classsizes})
and that
there are no vertices of degree larger than $\Delta_0$ (see~\ref{assump:Delta}), we have \whp
$$
e_{i,j} =
 n_i\cdot \bc{ \sum_{d \le \Delta_0} d\sum_{H \in \cS_d} \Pr\bc{\cT_i \cong H}
\pm \Delta_0 \cdot \dTV \bc{\empNBdist_{i,1}^\wpg,\cT_i}}
= n_i \bc{ \sum_{d \le \Delta_0} d\Pr\bc{\degdist_{ij}=d}+O\bc{\frac{\Delta_0}{\omega_0}}}
 = \Erw\bc{n_i}\bc{\Erw\bc{\degdist_{ij}}+O\bc{\frac{\Delta_0}{\omega_0}}}.
$$
By symmetry we also have $e_{i,j} =e_{j,i} =\Erw\bc{n_j}\bc{\Erw\bc{\degdist_{ji}}+O\bc{\frac{\Delta_0}{\omega_0}}}$.
It easily follows that $\Erw\bc{\degdist_{ij}} = 0 \Leftrightarrow \Erw\bc{\degdist_{ji}} = 0$,
in which case the statement follows trivially. On the other hand, if these expectations are non-zero, then we have
$\Erw\bc{\degdist_{ij}}+O\bc{\frac{\Delta_0}{\omega_0}} = \bc{1+O\bc{\frac{\Delta_0}{\omega_0}}}\Erw\bc{\degdist_{ij}}$,
and similarly for $\Erw\bc{\degdist_{ji}}$, so the result follows by rearranging.
\end{proof}

\section{Contiguity: Proof of Lemma~\ref{lem:contiguity}}\label{sec:contiguity}

The aim of this section is to prove Lemma~\ref{lem:contiguity}, the first of our two main steps,
which states that $\apg_{t_0}$ and $\wpg_{t_0}$ have approximately the same distribution.
We begin with an overview.

\subsection{Proof strategy}

The overall strategy for the proof is to show that every step of the construction of $\apg_{t_0}$
closely reflects the situation in $\wpg_{t_0}$. More precisely, the following are the critical steps in the proof.
Recall from Definition~\ref{def:wpg} that $\wpg$ is the underlying \emph{unmessaged} random graph corresponding
to $\wpg_{t_0}$, and similarly let
$\apg$ denote the underlying unmessaged random graph corresponding to $\apg_{t_0}$.
The following either follow directly from our assumptions or will be shown during the proof.
\newpage 

\begin{enumerate}
	\item The vectors representing the numbers of vertices of each type
	in $\apg_{t_0}$ and $\wpg_{t_0}$ are identically distributed.
	\item The local structure of $\wpg$ is described by the $\cT_i$ branching processes for $i \in [k]$. 
	\item After initialising Warning Propagation on $\wpg$ according to $\initialdistribution$ and proceeding for $t_0$ rounds,
	the distribution of the \instory\ along a random edge of type $ \bc{i,j}$ is approximately $\Matrdist{\vdistf{t_0}\bc{\initialdistribution}}{j,i}$.
	\item Given a particular \compseq, i.e.\ multiset of \stories\ (which consist of \instories\ and \outstories) on 
	half-edges at each vertex, 
	each graph with this \compseq\ is almost equally likely to be chosen as $\wpg$. 
	\item If we run Warning Propagation on $\apg$, with initialisation identical to the constructed $0$-messages in $\apg_{t_0}$,
	for $t_0$ steps, \whp\ the message histories are identical to those generated in the construction of $\apg_{t_0}$
	except on a set of $o\bc{n}$ edges.
\end{enumerate}

The first step is trivially true since we chose the vector $\vprobvec$ to be the distribution of the class sizes in $\wpg$.
The second step is simply~\ref{assump:LocS}, and
the third step is a direct consequence of the second
(see Proposition~\ref{prop:incominghistoryconcentration}).  One minor difficulty to overcome
in this step is how to handle the presence of short cycles, which are the main reason the approximations are not exact.  However,
since the local structure is a tree by~\ref{assump:LocS},
\whp\ there are few vertices which lie close to a short cycle (see Claim~\ref{claim:nearshortcycle}).

We will need to show that, while the presence of such a cycle close to a vertex may
alter the distribution of incoming message histories at this vertex (in particular they may no
longer be independent), it does not fundamentally alter which message histories are \emph{possible} (Proposition~\ref{prop:historyposprob}).
Therefore while the presence of a short cycle will change some distributions in its close vicinity,
the fact that there are very few short cycles means that this perturbation will be masked
by the overall random ``noise''.

The fourth step is precisely~\ref{assump:Eqlikely}, while
the fifth step is almost an elementary consequence of the fact that we constructed the message histories
in $\apg_{t_0}$ to be consistent with Warning Propagation (Proposition~\ref{prop:WPconsistent}).
In fact, it would be obviously true that \emph{all} message histories are identical were it not for the fact
that some half-edges may be left unmatched in the construction of $\apg$ and therefore deleted,
which can cause the \outmessages\ along other half-edges at this vertex to be incorrect.
This can then have a knock-on effect, but it turns out (see Proposition~\ref{prop:unmatched}) that \whp\ not too many
edges are affected.

\subsection{Plausibility of inputs}

We begin by showing that,
if we initialise messages in a (deterministic) graph in a way
which is admissible according to $\initialdistribution$,
any \inpt{t_0}\  at a half-edge of type $ \bc{i,j}$ produced by Warning Propagation
has a non-zero probability of appearing under the
probability distribution $\Matrdist{\vdistf{t_0}\bc{\initialdistribution}}{j,i}$.

\begin{proposition}\label{prop:historyposprob}
	Let $G$ be any $k$-type graph in which the type-degree of each vertex of type $i$ has positive
	probability under $\degdist_i$ and let $\bc{u,v}$ be a directed edge of $G$ of type $\bc{i,j}$.
	Suppose that messages are initialised in $G$ arbitrarily subject to the condition
	that each initial message is consistent with the vertex types and has non-zero probability under $\initialdistribution$, i.e.\ for every
	directed edge $\bc{u',v'}$ of type $\bc{i',j'}$, the
	initial message
	$\sigma \in \alphabet$ from $u'$ to $v'$ satisfies $g\bc{\sigma}=\bc{i',j'}$
	and furthermore $\Pr_{\initialdistribution}\bc{\sigma} \neq 0$. 
	Run Warning Propagation with update rule $\wpf$ for $t_0$ steps and 
		let $\muin := \messagehistory{u}{v}{t_0}$ and $\muout := \mu_{v\to u} \bc{0}$ be the resulting \instoryt{t_0} \ and  \outstoryt{0}
		at $v$ along $\bc{u,v}$ respectively.
	
	Then

		$$
		\Pr\Big(\Big( \Matrdist{\vdistf{t_0} \bc{\initialdistribution }}{i,j}, \Matrdist{\initialdistribution}{j,i} \Big)  = \bc{\muin, \muout } \Big)  \neq 0.
		$$
\end{proposition}

\begin{proof}
	We construct an auxiliary tree $G'$,
	in which each vertex has a corresponding vertex in $G$.
	For a vertex $w'$ in $G'$, the corresponding vertex in $G$ will be denoted by $w$.
	We construct $G'$ as follows.
	First generate $u'$ as the root of the tree, along with
	its parent~$v'$.
	Subsequently, recursively for each $t \in \{0\}\cup [t_0-1]$,
	for each vertex $x'$ at distance $t$ below $u'$ with parent $y'$, we generate
	children for all neighbours of the vertex $x$ in $G$
	except for $y$.
	
	Note that another way of viewing $G'$ is that we replace walks beginning
	at $u$ in $G$ (and whose second vertex is \emph{not} $v$)
	by paths, where two paths coincide for as long as the corresponding
	walks are identical, and are subsequently disjoint.
	A third point of view is to see $G'$ as a forgetful search tree of $G$,
	where (apart from the parent) we don't remember having seen vertices
	before and therefore keep generating new children.

	We will initialise messages in $G'$ from each vertex
	to its parent (and also from $v$ to $u$) according to the corresponding initialisation in $G$,
	and run Warning Propagation with update rule $\wpf$ for $t_0$ rounds.

		Let $ \muin' = {\vec \mu}_{u' \to v'}' \bc{\le t_0}$ be the resulting \instoryt{t_0}   and $\muout'= \mu'_{v' \to u'}\bc{0}$ be the \outstoryt{0} at $v'$
		 along $\bc{u',v'}$ in $G'$. Recall that $\muin$ and $\muout$ are the corresponding \instoryt{t_0} and \outstoryt{0} at $v$ in $G$ . The crucial observation is the following.

	\begin{claim}\label{claim:samehistories}
			$ \muin' = \muin$ and $\muout'=\muout$.
	\end{claim}

		We delay the proof of this claim until after the proof of Proposition~\ref{prop:historyposprob},
		which we now complete. Since each initial message has non-zero probability under $\initialdistribution$, we have
		$\Pr_{\initialdistribution}\bc{ \muout } \neq 0.$
		Recall that $\Matrdist{\vdistf{t_0}\bc{\initialdistribution}}{i,j}$ was defined as the probability distribution
		of $\bc{X_{ij}^{ \bc{0}},\ldots,X_{ij}^{ \bc{t_0}}}$, the message history in a $\cT_{ij}$ tree in which messages are
		initialised according to $\initialdistribution$.
		Therefore the probability that $\Matrdist{\vdistf{t_0}\bc{\initialdistribution}}{i,j}= \muin = \muin'$
		is certainly at least the probability that a $\cT_{ij}^{t_0}$ tree has exactly the structure of $G'$
		(up to depth $t_0$) and that the initialisation chosen at random according to $\initialdistribution$
		is precisely the same as the initialisation in $G'$. Since $G'$ is a finite graph
		whose type-degrees for all vertices not at distance $t_0$ from $u$ has positive probability
		under $\degdist$, there is a positive probability that a random instance of $\cT_{ij}^{t_0}$
		is isomorphic to $G'$.
		Furthermore, since each initial message has a positive probability under $\initialdistribution$,
		the probability of choosing the same initialisation as in $G'$ is also nonzero, as required.
\end{proof}

We now go on to prove the auxiliary claim.

\begin{proof}[Proof of Claim~\ref{claim:samehistories}]
	By construction the \outmessaget{0} at $v'$ along $\bc{v',u'}$ is identical to the corresponding \outmessaget{0} in $G$
	so $\muout' = \muout$. It remains to prove that the \instoriest{t_0} are identical.
	
	For any vertex $x' \in G'\setminus\{ v' \}$, let $x_+'$ denote the parent of $x'$.
	In order to prove Claim~\ref{claim:samehistories}, we will prove a much stronger statement from which the initial claim will follow easily.
	More precisely, we will prove by induction on $t$ that for all
	$x' \in G'\setminus\{ v' \}$,  $\MU'_{x' \rightarrow x_+'}\bc{\leq t} = \MU_{x \rightarrow {x_+}} \bc{\leq t}$. For $t=0$, by construction
	$\mu'_{x' \rightarrow x_+'} \bc{ 0 } = \mu_{x \rightarrow x_+ }\bc{0}$ for any $x' \in G'\setminus\{ v' \}$
	because messages in $G'$ are initialised according to the corresponding initialisation in $G$.
	Suppose that the statement is true  for some $t \leq t_0-1$.
It remains to prove that $\mu'_{x' \rightarrow x_+'}\bc{t+1} = \mu_{x \rightarrow x_+} \bc{t+1}$.
	By the induction hypothesis, $\mu'_{y' \rightarrow  x'} \bc{t} = \mu_{y \rightarrow  x} \bc{t}$ for all $y' \in \partial_{G'}x' \setminus \{ x_+'\}$.
	Hence,
	$$ \mset{ \mu'_{y' \rightarrow  x'} \bc{t} : y' \in \partial_{G'}x' \setminus \{ x_+'\} } =\mset{ \mu_{y \rightarrow  x}\bc{t} : y' \in \partial_{G'}x' \setminus \{ x_+'\} }
	=\mset{ \mu_{z \rightarrow  x}\bc{t} : z \in \partial_{G}x \setminus \{ {x_+}\} },
	$$
	i.e.\ the multisets of incoming messages to the directed edge $\bc{x',x_+'}$ in $G'$ and to the directed edge $\bc{x,x_+}$ in $G$ at time $t$ are identical.
	Therefore also
	$$
	\mu'_{x' \rightarrow x_+'}\bc{t+1} = \wpf\bc{ \mset{ \mu_{y' \rightarrow  x'}\bc{t} : y' \in \partial_{G}x' \setminus \{ x_+'\} } }
	= \wpf\bc{ \mset{ \mu_{z \rightarrow  x}\bc{t} : z \in \partial_{G}x \setminus \{ x_+\} } }=\mu_{x \rightarrow x_+}\bc{t+1},
	$$	
	as required.
\end{proof}

Proposition~\ref{prop:historyposprob} tells us that no matter how strange
or pathological a messaged graph looks locally, there is still a positive probability
that we will capture the resulting \inp\ (and therefore \whp\ such
an \inp\ will be generated a linear number of times in $\apg_{t_0}$).
In particular, within distance $t_0$ of a short cycle the distribution of an \inp\ 
may be significantly different from $\bc{\Matrdist{\vdistf{t_0}\bc{\initialdistribution}}{i,j},\Matrdist{\initialdistribution}{j,i}}$.
However, we next show that there are unlikely to be many edges this close to a short cycle.

\begin{claim}\label{claim:nearshortcycle}
	Let $W_0$ be the set of vertices which lie on some cycle of length at most $t_0$ in $\wpg $,
	and recursively define $W_t := W_{t-1} \cup \partial W_{t-1}$ for $t \in \NN$.
	
	Then \whp\ $\abs{W_{t_0}} = O\bc{\frac{n}{\omega_0}}$.
\end{claim}

\begin{proof}
Any vertex which lies in $W_{t_0}$ certainly has the property
that its neighbourhood to depth $2t_0$ contains a cycle.
However, since for any $i \in [k]$, the branching process
$\mathcal{T}_i^{2t_0}$ certainly does \emph{not} contain a cycle,
Assumption~\ref{assump:LocS} (together with the fact that \whp\ there
are $O\bc{n}$ vertices in total due to~\ref{PP:NConcentration})
shows that \whp\ at most $O\bc{n/\omega_0}$ vertices have such a cycle
in their depth $2t_0$ neighbourhoods. 
\end{proof}

\subsection{The deleted half-edges}

In the construction of $\apg$ we deleted some half-edges which remained unmatched in Step~\ref{def:apgMatch},
and it is vital to know that there are not very many such half-edges.
We therefore define $E_0$ to be the set of half-edges which are deleted in Step~\ref{def:apgMatch} of the random construction of $\apg$.

\begin{definition}
	Given integers $d,t \in \NN_0$, a messaged graph $G \in \histgraphs{n}{t_0}$ and a multiset $A \in \ms{\alphabet^{t+2}}{d}$,
	define $n_A= n_A\bc{G}$ to be the number of vertices of $G$ which receive \incompilation\ $A$.

	Further, let $\gamma_A^i=\gamma_A^i\bc{t}$ denote the
	probability that the \incompilationt{t}\ at a vertex of type $i$
	when generating $\apg_t$ is~$A$.
\end{definition}

Observe that for any $d,t \in \NN_0$, the expression $\sum_{A \in \ms{\alphabet^{t+2}}{d}} n_A \bc{G}$ is simply the number of vertices of degree $d$,
and therefore
for any $t\in \NN_0$ we have $\sum_{d \in \NN_0} \sum_{A \in \ms{\alphabet^{t+2}}{d}} n_A\bc{G}=\abs{V\bc{G}}$.

Recall that in Proposition~\ref{prop:Assume1}, apart from the function $F$ and the parameter $\omega_0$,
we also fixed parameters $c_0,d_0$, which we will now make use of. 
\begin{proposition}\label{prop:unmatched}
\Whp\ $\abs{E_0} = o\bc{\frac{n}{\sqrt{c_0}}}$.
\end{proposition}

\begin{proof}
Let us fix two \instoriest{t_0} $\vec \mu_1,\vec \mu_2 \in \alphabet^{t_0+1}$
and consider the number of half-edges $m_{\vec \mu_1,\vec \mu_2}$ with \instoryt{t_0} $\vec \mu_1$
and \outstoryt{t_0} $\vec \mu_2$.
We aim to show that $m_{\vec \mu_1,\vec \mu_2}$ is concentrated around its expectation
$\overline{m}_{\vec \mu_1,\vec \mu_2}$ as defined in~\eqref{eq:mbardef}.
Recall that the multiset of \storiest{t_0} at a vertex is determined by the
\incompilationt{t_0}, i.e.\ the multiset of \inpst{t_0}.
For each $d_1,d_2 \in \NN$, let $B_{d_1,d_2} = B_{d_1,d_2} \bc{\vec \mu_1,\vec \mu_2}$ denote the set of
\incompilationst{t_0} $A \in \ms{\alphabet^{t_0+2}}{d_2}$ consisting of $d_2$ many \inpst{t_0} which
lead to $d_1$ half-edges with \storyt{t_0} $\bc{\vec \mu_1,\vec \mu_2}$, and let $x_A$ denote
the number of vertices which receive \incompilationt{t_0} $A$ in Step~\ref{def:apgInstory} of the construction of $\apg_{t_0}$
(in particular \emph{before} the deletion of half-edges).
Then we have
\begin{align*}
m_{\vec \mu_1,\vec \mu_2} & = \sum_{d_1,d_2 \in \NN} \sum_{A \in B_{d_1,d_2}}d_1 x_A
\end{align*}
We split the sum into two cases, depending on $d_2$. Consider first the case when $d_2>d_0$.
By~\ref{PP:NConcentration} \whp\ the total number of vertices is $\Theta\bc{n}$,
and by~\ref{PP:FandZ} the probability that any vertex has degree larger than $d_0$ is at most
$1/F\bc{d_0}$, 
and it follows that \whp\ the number of half-edges attached to vertices of degree larger than $d_2$
is dominated by $d_2 \cdot \Bin\bc{\Theta\bc{n},\frac{1}{F\bc{d_2}}}$.
Thus the expected number of half-edges attached to such high degree vertices is at most
$$
\Theta\bc{1}\sum_{d_2 \ge d_0} \frac{d_2n}{F \bc{d_2}} = \Theta \bc{1} \frac{d_0 n}{F\bc{d_0}},
$$
Now by~\eqref{PP:choicec0} we have $F\bc{d_0}\gg c_0$ and also $d_0 \le \sqrt{\exp\bc{d_0}} \ll \sqrt{c_0}$, 
and therefore 
 $\frac{d_0 n}{F\bc{d_0}} = o\bc{\frac{n}{\sqrt{c_0}}}$.
An application of Markov's inequality shows that \whp\ the number of half-edges attached to vertices of degree at least $d_0$ is $o\bc{\frac{n}{\sqrt{c_0}}}$.

We now turn our attention to the case $d_2 \le d_0$. Here we observe that for any $A$ each vertex of $V_i$ is given
\incompilationt{t_0} $A$ with probability $\gamma_A^i $ independently, and so the number of vertices which receive $A$
is distributed as
$$
X:=\sum_{i=1}^k X_i = \sum_{i=1}^k \Bin\bc{n_i,\gamma_A^i}.
$$

Conditioning on the high probability event that $n_i = \bc{1+o\bc{\frac{1}{\omega_0}}}\Erw\bc{n_i}$
(see Claim~\ref{claim:classsizes}),
and in particular is $\Theta\bc{n}$,
a standard Chernoff bound shows that with probability at least $1- \exp\bc{ - \Theta\bc{ \bc{\ln n}^2} }$ the random variable $X$
is within an additive factor $\sqrt{n}\ln n$
of its expectation, 
and a union bound over all at most $\abs{\alphabet}^{ \bc{t_0+1}d_0} \stackrel{\text{\tiny \eqref{eq:secondchoicec0}}}{=} o\bc{c_0} \ll n^{1/5}$
choices for~$A$ of size at most~$d_0$
shows that \whp\ this holds for all such~$A$ simultaneously.

It follows that \whp\ 
\begin{align} \label{eq:concentrationmmu}
	\abs{m_{\vec \mu_1,\vec \mu_2} - \overline{m}_{\vec \mu_1,\vec \mu_2}}
	& \le \abs{m_{\vec \mu_1,\vec \mu_2} - \Erw\bc{\degdist_{g\bc{\vec \mu_1}}}q_{\vec \mu_1,\vec \mu_2} n_{g_1\bc{\vec \mu_1}}} + \abs{\Erw\bc{\degdist_{g\bc{\vec \mu_1}}}q_{\vec \mu_1,\vec \mu_2} n_{g_1\bc{\vec \mu_1}} - \overline{m}_{\vec \mu_1,\vec \mu_2}} \nonumber\\
	& \le \abs{\alphabet}^{ \bc{t_0+1}d_0}\sqrt{n}\ln n + o\bc{\frac{n}{\sqrt{c_0}}} + o\bc{\frac{n}{\omega_0}}
	= o\bc{\frac{n}{\sqrt{c_0}}},
\end{align}
To see the last estimate, note that
by~\eqref{eq:secondchoicec0} we have  $\abs{\alphabet}^{\bc{t_0+1}d_0}\sqrt{n}\ln n \ll c_0\sqrt{n}\ln n = o\bc{n/\sqrt{c_0}}$,
where second estimate follows since $c_0 \ll \omega_0 \ll n^{1/5}$ by~\ref{PP:omegadeltanpower} and~\ref{PP:choicec0}.
This last fact also implies that $\sqrt{c_0} \ll c_0 \ll \omega_0$.

Since this is true for any arbitrary \storiest{t_0} $\vec \mu_1,\vec \mu_2$, we can deduce that \whp\ 
$$
\abs{m_{\vec \mu_1,\vec \mu_2}-m_{\vec \mu_2,\vec \mu_1}}
= \abs{\overline{m}_{\vec \mu_1,\vec \mu_2}-\overline{m}_{\vec \mu_2,\vec \mu_1}} + o\bc{\frac{n}{\sqrt{c_0}}}.
$$
Moreover, by Claim~\ref{claim:ZijandNi} we have
$\abs{\barm_{\voone, \votwo} - \barm_{\votwo, \voone}} = O\bc{\frac{n\Delta_0}{\omega_0}}
\stackrel{\text{\tiny \ref{PP:choicec0}}}{=} o\bc{\frac{n}{\sqrt{c_0}}}$.
	Hence 
$
\abs{m_{\vec \mu_1,\vec \mu_2}-m_{\vec \mu_2,\vec \mu_1}}
= o\bc{\frac{n}{\sqrt{c_0}}},
$
and a union bound over all of the at most $\abs{\alphabet}^{2 \bc{t_0+1}} = O\bc{1}$ choices for $\vec \mu_1,\vec \mu_2$
implies that \whp\ the same is true for \emph{all} choices of $\vec \mu_1,\vec \mu_2$ simultaneously.

Finally, we observe that (deterministically) the number $\abs{E_0}$ of half-edges left unmatched is
$$
\abs{E_0}=\sum_{\vec \mu_1 \neq \vec \mu_2} \frac{1}{2} \abs{m_{\vec \mu_1,\vec \mu_2}-m_{\vec \mu_2,\vec \mu_1}}
+ \sum_{\vec \mu_1} \vecone\cbc{m_{\vec \mu_1,\vec \mu_1} \notin 2\NN}.
$$
The first term is $o\bc{\frac{n}{\sqrt{c_0}}}$ \whp\ by the arguments above, while the second term is deterministically at most
the number of $\vec \mu_1$ over which the sum ranges, which is at most $\abs{\alphabet}^{t_0+1}=O\bc{1}$.
Therefore \whp\ $\abs{E_0}=o\bc{\frac{n}{\sqrt{c_0}}}$, as required.\end{proof}

\subsection{Similar \incompilations}
Our next goal is to show that the \incompseq\ distribution in $\wpg_{t_0}$
is essentially the same as that in $\apg_{t_0}$.

\begin{proposition}\label{prop:incominghistoryconcentration}
		Let $t_0$ be some (bounded) integer. Then \whp\ the following holds.
		\begin{enumerate}
			\item
			For every integer $d \le d_0$ and for every $A \in \ms{\alphabet^{t_0+2}}{d}$ we have 
			$
			n_A \bc{\wpg_{t_0}}, n_A\bc{\apg_{t_0}} = \bc{\sum_{i \in [k]} \gamma_A^i n_i} +o\bc{\frac{n}{\sqrt{c_0}}}.
			$
			\item
			$\apg_{t_0},\wpg_{t_0}$ each contains at most $\frac{n}{c_0}$ vertices of degree at least $d_0$. 		\end{enumerate}
\end{proposition}

\begin{proof}
The proof is technical, but ultimately standard and we give only a short overview.
The proofs of the two statements for $\apg_{t_0}$ essentially already appear
in the proof of Proposition~\ref{prop:unmatched}, which estimated the same parameters
in the random model \emph{before} half-edges were deleted.
We therefore only need to
additionally take account of
the fact that some half-edges were deleted, but Proposition~\ref{prop:unmatched} itself
implies that this will not affect things too much. 

To prove the first statement for $\wpg_{t_0}$ we apply~\ref{assump:LocS}.
More precisely, the sets of local neighbourhoods up to depth $t_0$ in $\wpg$
of all vertices of $V_i$ look similar to $n_i$ independent copies of $\cT_i^{t_0}\bc{\degdist}$.
Furthermore, since the message initialisation in $\wpg$ is according to $\initialdistribution$,
and since there are very few dependencies between the local neighbourhoods, the same
is true if we consider the \emph{messaged} local neighbourhoods at time $0$.
Since these messaged neighbourhoods determine the corresponding \inpt{t_0} at the root,
a Chernoff bound shows that \whp\ we have concentration of $n_A\bc{\wpg_{t_0}}$
around its expectation. Importantly the $1/\omega_0$ term that describes the speed of convergence
of the local structure to $\cT_i^{t_0}$ is smaller than $1/\sqrt{c_0}$, the (normalised) error term in the statement.

For the second statement, we also apply~\ref{assump:LocS},
although here we only need to go to depth $1$ and need not consider any messages.
We also use~\ref{assump:Delta}
to bound the number of half-edges attached to vertices at which $\wpg$ and the
copies of $\mathcal{T}_i^1$ disagree.
Otherwise the proof is similar.
\end{proof}

Let $a_0:= \frac{\sqrt{c_0}}{4d_0\abs{ \alphabet}^{\bc{t_0+2}d_0}}$.
As a corollary of Proposition~\ref{prop:incominghistoryconcentration}, we obtain the following result.

\begin{corollary}\label{cor:similarhistories}
	After re-ordering vertices if necessary, \whp\ the number of vertices
	whose \incompilations\ are different in $\apg_{t_0}$ and $\wpg_{t_0}$
	is at most $\frac{n}{a_0}$.
\end{corollary}

\begin{proof}
Assuming the high probability event of Proposition~\ref{prop:incominghistoryconcentration} holds,
	the number of vertices with differing \incompilations\ is at most
\begin{align*}
\bc{\sum_{d=0}^{d_0}\sum_{A \in \ms{\alphabet^{t_0+2}}{d}} \frac{2n}{\sqrt{c_0}}} + \frac{2n}{c_0}
& \le \frac{2n}{ \sqrt{c_0}} \bc{\sum_{d=0}^{d_0} \abs{\alphabet}^{\bc{t_0+2}d}} + \frac{2n}{c_0}\\
& \le \frac{2n}{\sqrt{c_0}}  d_0\abs{\alphabet}^{\bc{t_0+2}d_0} + \frac{2n}{c_0} 
= \frac{2 n}{ 4a_0} + \frac{2n}{c_0}
\le \frac{n}{a_0},
\end{align*}
where the last approximation follows by definition of $a_0$.
\end{proof}

\subsection{Matching up}

Next, we show that choosing the random matching as we did in
Step~\ref{def:apgMatch} of
the construction of $\apg_{t_0}$
is an appropriate choice.
We already defined the type-degree sequence of a graph, which generalises the degree sequence,
but we need to generalise this notion still further to also track the in-coming \stories\ at a vertex.

\begin{definition}
	For any $\alphabet^{t_0+1}$-messaged graph $G\in \histgraphs{n}{t_0}$,
	let $H_i= H_i\bc{G}$ denote the \incompilation\ at vertex $i$, for $i \in [n]$
	and let $\vec H \bc{G}:= \bc{H_1,\ldots,H_n}$ be the \incompseq.
\end{definition}

\begin{claim}\label{claim:samehistseqsameprob}
	Suppose that $G_1,G_2$ are two graphs on $[n]$ with $\vec H\bc{G_1}=\vec H\bc{G_2}$.
	Then
	$
	\Pr\bc{\wpg=G_1} =\bc{1+o \bc{1}} \Pr\bc{\wpg=G_2}.
	$
\end{claim}
\begin{proof}
If $\vec H\bc{G_1}=\vec H\bc{G_2}$, then in particular $\vD\bc{G_1}=\vD\bc{G_2}$.  Then by Assumption~\ref{assump:Eqlikely}, we have that $\Pr\bc{\wpg=G_1}=\bc{1+o\bc{1}}\Pr\bc{\wpg=G_2}$. 
\end{proof}

\subsection{Message consistency}

We also need to know that the message histories generated in the construction of $\apg_{t_0}$
match those that would be produced by Warning Propagation.
Let $\apg_{\WP}$ denote the graph with message histories generated by constructing $\apg_{t_0}$,
stripping all the message histories except for the messages at time $0$ and running Warning Propagation
for $t_0$ steps with this initialisation. Furthermore, let $X_0$ be the set of vertices
at which some half-edges were deleted in Step~\ref{def:apgMatch} of the construction of $\apg_{t_0}$,
and for $t\in \NN$ let $X_t$ be the set of vertices at distance at most $t$ from $X_0$ in $\apg_{t_0}$.

\begin{proposition}\label{prop:WPconsistent}
	Deterministically we have $\apg_{\WP} = \apg_{t_0}$ except on those edges incident to $X_{t_0}$.
	Furthermore, on those edges incident to $X_{t_0}$ but not $X_{t_0-1}$, the message histories
	in $\apg_{\WP}$ and $\apg_{t_0}$ are identical up to time $t_0-1$.
\end{proposition}

\begin{proof}
	Since the two underlying unmessaged graphs are the same, we just need to prove that at any time $0 \le t \le t_0$,
	the incoming and outgoing messages at a given vertex $v\notin X_{t-1}$ are the same for $\apg_{t_0}$ and $\apg_{\WP}$
	(where we set $X_{-1}:=\emptyset$). We will prove the first statement by induction on $t$.
	At time $t=0$, the statement is true by construction of $\apg_{\WP}$.
	Now suppose it is true up to time $t$ for some $0\le t \le t_0-1$ and consider an arbitrary directed edge $\bc{u,v}$ between vertices $u,v \notin X_t$.
	By Definition~\ref{def:apg}~$ \bc{2}$, the \outmessaget{(t+1)} from $u$ in $\apg_{t_0}$ is produced according to the rules of
	Warning Propagation based on the \inmessagest{t} to $u$ at time $t$.  Since $u \notin X_t$,
	none of its neighbours lie in $X_{t-1}$ and therefore by the induction hypothesis, these \inmessagest{t} are the same
	for $\apg_{t_0}$ and $\apg_{\WP}$. Hence, the \outmessaget{(t+1)} along $\bc{u,v}$ is also the same in $\apg_{t_0}$ and $\apg_{\WP}$.
This proves the first statement of the proposition, while the second follows from the inductive statement for $t=t_0-1$.
\end{proof}

In view of Proposition~\ref{prop:WPconsistent}, we need to know that not too many edges are incident to $X_{t_0}$.

\begin{proposition}
Let $t\in \NN$ be any constant. \Whp\ the number of edges of $\apg$ incident to $X_t$ is $o\bc{n}$.
\end{proposition}

\begin{proof}
The statement for $t=0$ is implied by the (slightly stronger) statement of Proposition~\ref{prop:unmatched}.
For general $t$, the statement follows since the average degree in $\apg$ is bounded.
More precisely, the expected number of edges of $\apg$ incident to $X_t$ is $\bc{O \bc{1}}^t \abs{X_0} = O\bc{1}\abs{X_0} = o\bc{n}$,
and an application of Markov's inequality completes the proof.
\end{proof}

\subsection{Final steps}

We can now complete the proof of Lemma~\ref{lem:contiguity}.

\begin{proof}[Proof of Lemma~\ref{lem:contiguity}] 
	We use the preceding auxiliary results to show that every step in the construction
	of $\apg_{t_0}$ closely mirrors a corresponding step in which we reveal partial information about
	$\wpg_{t_0}$.
	Let us first explicitly define these steps within $\wpg_{t_0}$ by revealing information
	one step at a time as follows.
	\begin{enumerate}
		\item First reveal the \incompilation\ at each vertex, modelled along half-edges.
		\item Next reveal all \outstories\ along each half-edge.
		\item Finally, reveal which half-edges together form an edge.
	\end{enumerate}
	
	Corollary~\ref{cor:similarhistories} shows that Step~\ref{def:apgHalf} in the construction of
	$\apg_{t_0}$ can be coupled
	with Step~\ref{def:apgHalf} in revealing $\wpg_{t_0}$ above in such a way that \whp\ the number of vertices
	on which they produce different results is at most $\frac{n}{a_0} = o\bc{n}$.
	Furthermore, Proposition~\ref{prop:WPconsistent} shows that,
	for those vertices for which the \incompilations\ are identical in Step~\ref{def:apgHalf},
	the \outstories\ generated in Step~\ref{def:apgInstory} of
	the construction of both $\apg_{t_0}$ and $\wpg_{t_0}$ must also be identical (deterministically).
	Therefore before the deletion of unmatched half-edges in Step~\ref{def:apgMatch} of the definition of $\apg_{t_0}$,
	\whp\ Condition~$\bc{2}$ of Definition~\ref{def:graphsclose} is satisfied.
	On the other hand, Proposition~\ref{prop:unmatched} states that \whp\ $o\bc{n/\sqrt{c_0}}=o\bc{n}$ half-edges are deleted,
	and therefore the condition remains true even after this deletion.
	
	Now in order to prove that we can couple the two models in such a way that the two edge sets
	are almost the same (and therefore Condition~$\bc{1}$ of Definition~\ref{def:graphsclose} is satisfied),
	we consider each potential \story\ $\vom \in \alphabet^{2\bc{t_0+1}}$
	in turn, and construct coupled random matchings of the corresponding half-edges.
	More precisely, let us fix $\vom$ and let $\apm$ be the number of half-edges with this \story\ 
	in $\apg_{t_0}$.
	Similarly, define $\wpm$ to be the corresponding number of half-edges in $\wpg_{t_0}$.
	Furthermore, let $\apr_1$ be the number of half-edges with \story\ $\vom$ in $\apg_{t_0}\setminus \wpg_{t_0}$,
	let $\apr_2$ be the number of half-edges with the ``dual \story'' $\vom^*$, i.e.\ the \story\ with \instory\ and \outstory\ switched,
	and correspondingly $\wpr_1,\wpr_2$ in $\wpg_{t_0}\setminus \apg_{t_0}$.
	
	For convenience, we will assume that $\vom^* \neq \vom$; the case
	when they are equal is very similar.

	Let us call an edge of a matching \emph{good} if it runs between two half-edges which
	are common to both models. Note that this does not necessarily mean it is common to both matchings,
	although we aim to show that we can couple in such a way that this is (mostly) the case.
	Observe that, conditioned on the number of good edges in a matching, we may
	first choose a matching of this size uniformly at random on the common half-edges,
	and then complete the matching uniformly at random (subject to the condition that we never match two
	common half-edges).
	
	Observe further that the matching in $\apg_{t_0}$ must involve at least $\apm-\apr_1-\apr_2$ good edges,
	and similarly the matching in $\wpg_{t_0}$ must involve at least $\wpm-\wpr_1-\wpr_2$,
	and therefore we can couple in such a way that at least
	$\min\{\apm-\apr_1-\apr_2,\wpm-\wpr_1-\wpr_2\}$ edges are identical, or in other words, the symmetric difference
	of the matchings has size at most $\max\{\apr_1+\apr_2,\wpr_1+\wpr_2\}$.
	
	Repeating this for each possible $\vom$, the total number of edges in the symmetric
	difference is at most twice the number of half-edges which are not common to both models.
	We have already shown that there are at most $o\bc{\frac{n}{a_0}} + o\bc{\frac{n}{\sqrt{c_0}}} = o\bc{\frac{n}{a_0}}$ vertices
	at which the \incompilations\ differ, and applying the second statement of Proposition~\ref{prop:incominghistoryconcentration},
	we deduce that \whp the number of half-edges which are not common to both models is at most
	$
	d_0 \cdot o\bc{\frac{n}{a_0}} + 2\cdot \frac{n}{c_0} = o\bc{n}
	$
	as required.
\end{proof}

\section{Subcriticality: The idealised change process}\label{sec:subcritical}

With Lemma~\ref{lem:contiguity} to hand, which tells us that $\wpg_{t_0}$ and $\apg_{t_0}$ look very similar,
we break
the rest of the proof of Theorem~\ref{thm:main} down into two further steps.

First, in this section, we describe an idealised approximation of how a change propagates
when applying $\WP$ repeatedly to $\proj{\wpg}_{t_0}$, and show that this approximation
is a subcritical process, and therefore quickly dies out. The definition of this idealised change process is motivated
by the similarity to $\apg_{t_0}$.

In the second step, in Section~\ref{sec:beyond} we will use Lemma~\ref{lem:contiguity} to prove formally that
the idealised change process closely approximates the actual change process, which therefore also quickly terminates.

\begin{definition}
	Given a probability distribution matrix $Q$ on $\alphabet$,
	we say that a pair of messages $\bc{\changeold{0},\changenew{0}}$ is a \emph{potential change} with respect to $Q$
	if there exist some $t \in \NN$ and some $\vom = \bc{\mu_0,\mu_1,\ldots,\mu_t} \in \Consist{t+1}$
	such that
	\begin{itemize}
		\item $\mu_{t-1}=\changeold{0}$;
		\item $\mu_t = \changenew{0}$;
		\item $\Pr\bc{\Matrdist{\vdistf{t}\bc{Q}}{\switch{g}\bc{\mu}} = \vom}>0$.
	\end{itemize}
	We denote the set of potential changes by $\potentialchanges\bc{Q}$.
\end{definition}

In other words, $\bc{\changeold{0},\changenew{0}}$ is a potential change
if there is a positive probability of making a change from $\changeold{0}$ to $\changenew{0}$ in the message at the root edge 
at some point in the Warning Propagation algorithm on a $\cT_{g\bc{\changeold{0}}}$ branching tree
when initialising according to $Q$.
The following simple claim will be important later.

\begin{claim}\label{claim:changeposprob}
	If $P$ is a fixed point and $\bc{\changeold{0},\changenew{0}} \in \potentialchanges\bc{P}$ with $g\bc{ \changeold{0}}=\bc{i,j}$, then $\Matrdist{P}{i,j}\bc{\changeold{0}}>0$ and $\Matrdist{P}{i, j}\bc{\changenew{0}}<1$.
\end{claim}

\begin{proof}
	The definition of $\potentialchanges\bc{P}$ implies in particular that there exist a
	$t \in \NN$ and a $\vom\in \Consist{t+1}$  such that $\mu_{t-1} = \changeold{0}$
	and $\Pr\bc{\Matrdist{\vdistf{t}\bc{P}}{i, j}=\vom}>0$. Furthermore, by Claim~\ref{claim:localhistorydistribution},
	the marginal distribution of the $t$-th entry of $\Matrdist{\vdistf{t}\bc{P}}{i, j}$ is $\Matrdist{\distf^{t}\bc{P}}{i, j}=\Matrdist{P}{i, j}$ (since $P$ is a fixed point),
	and therefore we have $\Matrdist{P}{i,j}\bc{\changeold{0}} \ge \Pr\bc{\Matrdist{\vdistf{t}\bc{P}}{i, j}=\vom}>0$.
	
	On the other hand, since $\Matrdist{P}{i,j}$ is a probability distribution on $\alphabet$,
	clearly $\Matrdist{P}{i,j}\bc{\changenew{0}} \le 1- \Matrdist{P}{i,j}\bc{\changeold{0}}<1$.
\end{proof}

\subsection{The idealised change branching process}

Given a probability distribution matrix $Q$ on $\alphabet$ and a pair $\bc{\changeold{0}, \changenew{0}} \in \potentialchanges\bc{Q}$,
	we define a branching process
	$\changebranch = \changebranch \bc{\changeold{0},\changenew{0},Q}$ as follows. 
We generate an instance of $\cT_{ij}$, where $\bc{i,j}= \bar g\bc{\changeold{0}}$,
in particular including messages upwards to the directed root edge $\bc{v,u}$, so $u$ is the parent of $v$.
We then also initialise two messages downwards along this root edge, $\mu^{\bc{1}}_{u\to v} = \changeold{0}$ and $\mu^{\bc{2}}_{u\to v} = \changenew{0}$.
We track further messages down the tree based on the message that a vertex receives from its parent
and its children according to the WP update rule~$\wpf$.
Given a vertex $y$ with parent $x$, let $\mu^{\bc{1}}_{x \to y}$ be the resultant message
when the input at the root edge is $\mu^{\bc{1}}_{u\to v} = \changeold{0}$,
and similarly $\mu^{\bc{2}}_{x\to y}$ the resultant message when the input is $\mu^{\bc{2}}_{u\to v} = \changenew{0}$.
Finally, delete all edges $\bc{x,y}$ for which $\mu^{\bc{1}}_{x\to y}=\mu^{\bc{2}}_{x\to y}$, so
we keep only edges at which messages change (along with any subsequently isolated vertices).
It is an elementary consequence of the construction that $\changebranch$ is necessarily a tree.

\subsection{Subcriticality}

Intuitively, $\changebranch$ approximates the cascade effect that a single change
in a message from time $t_0-1$ to time $t_0$ subsequently causes (this is proved
more precisely in Section~\ref{sec:beyond}).
Therefore while much of this paper is devoted to showing that
$\changebranch$ is indeed a good approximation,
a very necessary task albeit an intuitively natural outcome,
the following result is the essential heart of the
proof of Theorem~\ref{thm:main}.

\begin{proposition}\label{prop:subcritical}
	If $P$ is a stable fixed point, then
	for any $\bc{\changeold{0},\changenew{0}} \in \potentialchanges\bc{P}$,
	the branching process $\changebranch = \changebranch \bc{\changeold{0},\changenew{0},P}$ is subcritical.
\end{proposition}

\begin{proof}
Let us suppose for a contradiction that for some $\bc{\changeold{0},\changenew{0}}\in \potentialchanges\bc{P}$,
the branching process has survival probability $\rho >0$.
We will use the notation $a \muchless b$
to indicate that given $b$, we choose $a$ sufficiently small as a function of $b$.\footnote{In the literature
this is often denoted by $a \ll b$, but we avoid this notation since it has a very different meaning elsewhere in the paper.
In particular, here we aim to fix several parameters which are all constants rather than functions in $n$.}

Given $\rho$ and also $\alphabet,\wpf,P$, let us fix further parameters $\eps,\delta \in \RR$ and $t_1 \in \NN$
according to the following hierarchy:
$$
0<\eps \muchless \frac{1}{t_1} \muchless \delta \muchless \rho,\frac{1}{\abs{\alphabet}} \le 1.
$$
In the following, given an integer $t$ and messages $\changeold{t},\changenew{t}\in \alphabet$,
we will use the notation $\changeboth{t}:= \bc{\changeold{t},\changenew{t}}$.
 Let us define a new probability distribution matrix  $Q$ on $\alphabet$ as follows. For each $\bc{i,j} \in [k]^2$ and for all $\mu \in \alphabet$ 
$$
\Matrdist{Q}{i,j}\bc{\mu} :=
\begin{cases}
\Matrdist{P}{i,j}\bc{\mu} -\eps & \mbox{if } \bc{i,j}= g\bc{\changeboth{0}} \mbox{ and } \mu=\changeold{0};\\
\Matrdist{P}{i,j}\bc{\mu} +\eps &  \mbox{if } \bc{i,j}= g\bc{\changeboth{0}} \mbox{ and } \mu=\changenew{0}; \\
\Matrdist{P}{i,j} \bc{\mu} & \mbox{otherwise.}
\end{cases}
$$
In other words, we edit the probability distribution in the $g\bc{\changeboth{0}}$ entry of the matrix $P$ to shift some weight from $\changeold{0}$
to $\changenew{0}$, but otherwise leave everything unchanged.
Note that since $\bc{\changeold{0},\changenew{0}} \in \potentialchanges\bc{P}$ is a potential change,
for sufficiently small $\eps$, each entry $\Matrdist{Q}{i,j}$ of $Q$ is indeed a probability distribution
(by Claim~\ref{claim:changeposprob} for $\bc{i,j}=g\bc{\changeboth{0}}$ or trivially otherwise).

Let us generate the $t_1$-neighbourhood of a root vertex $u$ of type $i$ in a $\cT_{i}$
branching process and initialise messages from the leaves at depth $t_1$ according to both
$Q$ and $P$, where we couple in the obvious way so that all messages are
identical except for some which are $\changeold{0}$ under $P$ and $\changenew{0}$ under $Q$.
We call such messages \emph{changed} messages.

We first track the messages where we initialise with $P$ through the tree (both up and down) according to the Warning Propagation rules,
but without ever updating a message once it has been generated.
Since $P$ is a fixed point of $\wpf$,  each message $\mu$ either up or down in the tree has the distribution $\Matrdist{P}{g\bc{\mu}}$ (though clearly far from independently).

We then track the messages with initialisation according to $Q$ through the tree,
and in particular track where differences from the first set of messages occur.
Let $x_s\bc{\changeboth{1}}$ denote the probability that a message from a vertex at level
$t_1-s$ to its parent changes from $\changeold{1}$ to $\changenew{1}$.
Thus in particular we have
$$
x_0\bc{\changeboth{1}} =
\begin{cases}
\eps & \mbox{if } \changeboth{1} = \changeboth{0},\\
0 & \mbox{otherwise.}
\end{cases}
$$

Observe also that messages coming down from parent to child
``don't have time'' to change before we consider the message up (the changes from
below arrive before the changes from above). Since we are most interested in changes
which are passed \emph{up} the tree, we may therefore always consider a message coming
down as being distributed according to $P$
(more precisely, according to $\Matrdist{P}{i,j}$, where $i,j$ are the types of the parent
and child respectively).

We aim to approximate $x_{s+1}\bc{\changeboth{1}}$ based on $x_s$, so let us consider a vertex
$u$ at level $t_1-\bc{s+1}$ and its parent $v$.
Let us define $C_d=C_d\bc{u}$ to be the event that $u$ has precisely $d$ children.
Furthermore, let us define $D_u\bc{\changeboth{2}}$ to be the event
that exactly one change is passed up to $u$ from its children, and that this change is of type $\changeboth{2}$.
Finally, let $b_u\bc{\changeboth{1}}$ be the number of messages
from $u$ (either up or down) which change from $\changeold{1}$ to $\changenew{1}$ (there may be more changes of other types).

The crucial observation is that given the neighbours of $u$ and their types, each is equally likely to be the parent --
this is because the tree $\cT_i$ is constructed in such a way that, conditioned on the presence and type of the parent,
the type-degree distribution of a vertex of type $j$ is $\degdist_j$, regardless of what the type of the parent was.
 Therefore conditioned on the event $D_u\bc{\changeboth{2}}$ and
the values of $d$ and $b_u\bc{\changeboth{1}}$, apart from the one child from
which a change of type $\changeboth{2}$ arrives at $u$, there are $d$ other neighbours
which could be the parent, of which $b_u\bc{\changeboth{1}}$ will receive a change of type
$\changeboth{1}$.
Thus the probability that
a change of type $\changeboth{1}$ is passed up to the parent is precisely
$\frac{b_u\bc{\changeboth{1}}}{d}$.

Therefore in total, conditioned on $C_d$ and $D_u\bc{\changeboth{2}}$,
the probability $a_{d;\changeboth{1},\changeboth{2}}$ that a change of type $\changeboth{1}$ is passed on from $u$ to $v$
is
\begin{align*}
	a_{d;\changeboth{1},\changeboth{2}}
	& = \sum_{\ell=1}^d  \bc{\Pr\bc{b_u\bc{\changeboth{1}}=\ell \; | \; C_d \wedge D_u\bc{\changeboth{2}}} \cdot \frac{\ell}{d}}
	= \frac{1}{d} \cdot \Erw\bc{ b_u\bc{\changeboth{1}}  \; | \; C_d \wedge D_u\bc{\changeboth{2}}}.
	\end{align*}
Now observe that this conditional expectation term is exactly as in the change process.
More precisely, in the $\changebranch$ process we know automatically that only one change
arrives at a vertex, and therefore if we have a change of type $\changeboth{2}$, the event
$D_u \bc{\changeboth{2}}$ certainly holds. 
Therefore, letting $h=g_1\bc{\changeold{2}}$ and $\ell=g_2\bc{ \changeold{2} }$,
\begin{equation}\label{eq:transitionmatrixentry}
\sum_{d\ge 1} \sum_{\substack{ \vec{d} \in \Seq\bc{d} }} \Pr\bc{ \offdist{\ell}{h}=\vec{d} } d a_{d;\changeboth{1},\changeboth{2}}
= \Matrdist{T}{\changeboth{1},\changeboth{2}},
\end{equation}
where $\Seq\bc{d}$ is the set of sequences $\vec{d}:=\bc{d_1, \ldots, d_k} \in \NN_0^k$ such that $\sum_{\ell'=1}^{k} d_{\ell'}=d$ and
$T$ is the $\abs{\alphabet}^2 \times \abs{\alphabet}^2$ transition matrix associated with the $\changebranch$ change process,
i.e.\ the entry $\Matrdist{T}{\changeboth{1},\changeboth{2}}$ is equal to the expected number of changes of type
$\changeboth{1}$ produced in the next generation by a change of type $\changeboth{2}$.

On the other hand, defining $E_u$ to be the event that at least two children of $u$ send changed messages (of any type) to $u$, we also have
\begin{align}\label{eq:recursionapprox}
x_{s+1}\bc{\changeboth{1}} & \ge \sum_{d\ge 1} \sum_{ \vec{d} \in \Seq\bc{d}}\Pr\bc{\offdist{\ell}{h} = \vec{d}  } \sum_{\changeboth{2}\in \alphabet^2} a_{d;\changeboth{1},\changeboth{2}} \Pr\bc{D_u\bc{\changeboth{2}} | C_d} \nonumber \\
& \ge \sum_{d\ge 1} \sum_{\vec{d} \in \Seq\bc{d}}\Pr\bc{\offdist{\ell}{h} = \vec{d}  } \sum_{\changeboth{2}\in \alphabet^2} a_{d;\changeboth{1},\changeboth{2}}
\Big(dx_{s}\bc{\changeboth{2}} - \Pr \bc{E_u | C_d}\Big).
\end{align}
For each $s \in \NN$,
let $\vec x_s$ be the $\abs{\alphabet}^2$-dimensional vector whose
entries are $x_s\bc{\changes}$ for $\changes \in \alphabet^2$ (in some arbitrary order).
We now observe that, since $P$ is a stable fixed point, i.e.\ $\distf$ is a contraction on a neighbourhood of $P$,
and since $\dTV\bc{P,Q}=\eps$,
for small enough $\eps$ we have
$$
\sum_{\changes \in \alphabet^2} x_s\bc{\changes} = \| \vec x_s \|_1 = \dTV\bc{ P,\distf^s \bc{Q}} \le \dTV \bc{P,Q} = \| \vec x_0 \|_1 = \eps,
$$
and so we further have
\begin{equation}\label{eq:twochangeprob}
\Pr\bc{E_u | C_d} \le \binom{d}{2} \eps^2 \le d^2 \eps^2.
\end{equation}
Furthermore, we observe that since $a_{d;\changeboth{1},\changeboth{2}}$ is a probability term by definition, we have
\begin{equation}\label{eq:transitionsum}
\sum_{\changeboth{2}\in \alphabet^2}a_{d;\changeboth{1},\changeboth{2}} \le \sum_{\changeboth{2}\in \alphabet^2} 1 = \abs{\alphabet}^2.
\end{equation}
Substituting~\eqref{eq:transitionmatrixentry},~\eqref{eq:twochangeprob} and~\eqref{eq:transitionsum} into~\eqref{eq:recursionapprox}, we obtain
\begin{align*}
x_{s+1}\bc{\changeboth{1}}
&\ge  \sum_{\changeboth{2}\in \alphabet^2}\Matrdist{T}{\changeboth{1},\changeboth{2}}x_{s}\bc{\changeboth{2}}
- \abs{\alphabet}^2 \eps^2 \sum_{d\ge 1} d^2 \sum_{\vec{d} \in \Seq\bc{d}}\Pr\bc{\offdist{\ell}{h} = \vec{d}  } .
\end{align*}
Moreover, we have
\begin{align*}
	\sum_{d\ge 1} d^2 \sum_{\vec{d} \in \Seq\bc{d}}\Pr\bc{\offdist{\ell}{h} = \vec{d}  } = \sum_{d\ge 1} d^2 \Pr\bc{ \norm{\offdist{\ell}{h}}_1 = d  } =\Erw\bc{\norm{\offdist{\ell}{h}}_1^2 }.\\
\end{align*}

Now for any $h,\ell \in [k]$ we have that $\Erw\bc{\norm{\offdist{\ell}{h}}_1^2 }$ is finite by Remark~\ref{rem:finitemoments},
so defining $c:= \max_{h,\ell \in [k]} \Erw\bc{\norm{\offdist{\ell}{h}}_1^2 }$, we have
$$\abs{\alphabet}
\vec x_{s+1} \ge T \vec x_s -  c \abs{\alphabet}^2\eps^2
$$
(where the inequality is pointwise). As a direct consequence we also have
$
\vec x_{s} \ge T^s \vec x_0 - s c \abs{\alphabet}^2\eps^2
$
(pointwise), and therefore
$$
\| \vec x_s \|_1 \ge \|T^s\vec x_0\|_1 - s c \abs{\alphabet}^4\eps^2.
$$
Now since the change process has survival probability $\rho >0$
for the appropriate choice of $\changeboth{0}=\bc{\changeold{0},\changenew{0}}$, choosing $\vec x_0 = \eps \vec e_{\changeboth{0}}$
(where $\vec e_{\changeboth{0}}$ is the corresponding standard basis vector)
we have
$$
\| \vec x_{s} \|_1 \ge \| T^{s} \vec x_0 \|_1 - s c \abs{\alphabet}^4\eps^2 \ge \rho \|\vec x_0\|_1 - s c \abs{\alphabet}^4 \eps^2
=  \eps\bc{\rho - s c \abs{\alphabet}^4 \eps}.
$$
On the other hand, since $P$ is a stable fixed point,
there exists some $\delta>0$ such that
for small enough $\eps$ we have $\| \vec x_s \|_1 \le \bc{1-\delta}^s \eps$ for all $s$. In particular choosing
$s=t_1$,
we conclude that
$$
\eps\bc{\rho - t_1 c \abs{\alphabet}^4 \eps} \le \| \vec x_{t_1} \|_1 \le \bc{1-\delta}^{t_1} \eps. 
$$
However, since we have $\eps \muchless 1/t_1 \muchless \delta \muchless \rho,1/\abs{\alphabet}$,
we observe that
$$
\bc{1-\delta}^{t_1} \le \rho/2 < \rho - t_1 c \abs{\alphabet}^4\eps,
$$
which is clearly a contradiction.
\end{proof}

\section{Applying subcriticality: Proof of Theorem~\ref{thm:main}}\label{sec:beyond}

Our goal in this section is to use Proposition~\ref{prop:subcritical} to complete the proof of Theorem~\ref{thm:main}.

\subsection{A consequence of subcriticality}

	Recall that during the proof of Proposition~\ref{prop:subcritical} we defined the transition matrix
	$T$ of the change process $\changebranch$, which is a $\abs{\alphabet}^2 \times \abs{\alphabet}^2$ matrix where the entry
	$\Matrdist{T}{\changeboth{1},\changeboth{2}}$ is equal to the expected number of changes
	of type $\changeboth{1}$ that arise from a change of type $\changeboth{2}$.
	The subcriticality of the branching process is equivalent to $T^n \xrightarrow{n\to \infty} 0$
	(meaning the zero matrix), which is also equivalent to all eigenvalues of $T$ being strictly less than $1$ (in absolute value).
	We therefore obtain the following corollary of Proposition~\ref{prop:subcritical}.

\begin{corollary}\label{cor:subeigenvector}
	There exist a constant $\gamma>0$ and a positive real $\abs{\alphabet}^2$-dimensional vector $\vec \alpha$ (with no zero entries)
	such that
	$$
	T\vec \alpha \le \bc{1-\gamma}\vec \alpha
	$$
	(where the inequality is understood pointwise). We may further assume that $\| \vec \alpha \|_1 =1$.
\end{corollary}

\begin{proof}
	Given some $\epsilon>0$, let $T' = T'\bc{\epsilon}$ be the matrix obtained from $T$ by adding $\epsilon$ to each entry.
	Thus $T'$ is a strictly positive real matrix and we may choose $\epsilon$ to be small enough such that all the
	eigenvalues of $T'$ are still less than 1 in absolute value.
	By the Perron-Frobenius theorem, there exists a positive real eigenvalue that matches the spectral radius $\rho\bc{T'}<1$ of $T'$.
	In addition, there exists a corresponding eigenvector to $\rho\bc{T'}$,
	say $\vec \alpha$, all of whose entries are non-negative;
	since every entry of $T'$ is strictly positive, it follows that in fact every entry of $\vec \alpha$ is also strictly positive.
	We have $T'\vec \alpha = \rho\bc{T'}\vec \alpha$, and
	we also note that $T\vec \alpha < T'\vec \alpha$ since every entry of $T'$ is strictly greater than the corresponding entry of $T$.
	Thus we deduce that $T\vec\alpha < \rho\bc{T'}\vec\alpha$, and setting $\gamma := 1-\rho\bc{T'} >0$, we have the desired statement.
	
	The final property that $\|\vec \alpha \|_1=1$ can be achieved simply through scaling by an appropriate (positive) normalising constant,
	which does not affect any of the other properties of $\vec \alpha$.
\end{proof}

	However, let us observe that in fact the change process that we want to consider is slightly different -- rather than
	having \inmessages\ distributed according to $P$, they should be distributed according to $\distf^{t_0-1}\bc{\initialdistribution}$.
	Since $P$ is the stable limit of $\initialdistribution$, this is arbitrarily close, but not exactly equal, to $P$.
	We therefore need the following.

\begin{corollary}\label{cor:subeigenvector2}
		There exists $\delta_0>0$ sufficiently small that for any probability distribution
		$Q$ on $\alphabet$ which satisfies $\dTV\bc{P, Q} \le\delta_0$, the following holds.
		Let $\changebranch_1= \changebranch\bc{\changeold{0},\changenew{0},Q}$
		and let  $T_1$ be the transition matrix of $\changebranch_1$.
		Then
		there exist a constant $\gamma>0$ and a positive real $\abs{\alphabet}^2$-dimensional vector $\vec \alpha$ (with no zero entries)
		such that
		$$
		T_1\vec \alpha \le \bc{1-\gamma}\vec \alpha
		$$
		(where the inequality is understood pointwise).
\end{corollary}

	In other words, the same statement holds for $T_1$,
	the transition matrix of this slightly perturbed process, as for~$T$. In particular,
	$\changebranch_1$ is also a subcritical branching process.

\begin{proof}
		Observe that since
		$\dTV\bc{P, Q}\le \delta_0$, for any $\epsilon$ we may pick $\delta_0=\delta\bc{\epsilon}$ sufficiently small such that
		$T_1$ and $T$ differ by at most $\epsilon$ in each entry. In other words, we have
		$T_1 \le T'$ pointwise, where $T'=T'\bc{\epsilon}$ is as defined in the proof of Corollary~\ref{cor:subeigenvector}.
		Thus we also have $T_1\vec \alpha \le T' \vec \alpha = \rho\bc{T'}\vec \alpha = \bc{1-\gamma} \vec \alpha$
		as in the previous proof.
\end{proof}

For the rest of the proof, let us fix
$\delta$ as in Theorem~\ref{thm:main} and a constant $\delta_0 \muchless \delta$ small enough that
the conclusion of Corollary~\ref{cor:subeigenvector2} holds,
and also such that \whp\ $\sum_{i=1}^k n_i \le \delta_0^{-1/100} n$,
 which is possible because by Claim \ref{claim:classsizes}
we have $n_i = \bc{1+o \bc{1}}\Erw\bc{ n_i}=\Theta\bc{n}$ \whp.
Moreover, suppose that $t_0$ is large enough that $P':=\distf^{t_0-1}\bc{\initialdistribution}$
satisfies $\dTV \bc{P, P'} \le \delta_0$ (this is possible since
$\distf^{*}\bc{\initialdistribution}=P$).

\subsection{The marking process}

We now use the idealised form $\changebranch_1$ of the change process to give an upper bound on the (slightly messier) actual process.
For an upper bound, we will slightly simplify the process of changes made by $\WP$ to obtain
$\WP^*\bc{\proj{\wpg}_{t_0}} = \WP^*\bc{\proj{\wpg}_0}$ from $\proj{\wpg}_{t_0}$.
\footnote{Note here that with a slight abuse of notation, we use $\WP$ to denote the obvious function on $\messgraphs{n}$
	which, given a graph $G$ with messages $\vec \mu \in \msgspace \bc{G}$,
	maps $\bc{G,\vec \mu}$ to $\WP\bc{G,\vec \mu}:= \bc{G,\WP_G\bc{\vec \mu}}$.}

We will reveal the information in $\wpg_{t_0}$ a little at a time as needed.

\begin{itemize}
	\item Initialisation
	\begin{itemize}
		\item We first reveal the \inpst{t_0} at each vertex,
		and the corresponding \outstories\ according to the update rule $\wpf$.
		We also generate the outgoing messages at time $t_0+1$.
		Any half-edge whose \outmessaget{t_0} is $\changeold{0}$ and whose \outmessaget{(t_0+1)} is $\changenew{0}\neq \changeold{0}$
		is called a \emph{change of type $\changeboth{0}$}.
		
		\item For each \outstory\ which includes a change, this half-edge is \emph{marked}.
	\end{itemize}
	\item We continue with a \emph{marking process}:
	\begin{itemize}
		\item Whenever a half-edge at $u$ is marked, we reveal its partner $v$. The edge $uv$ is marked.
		\item If $v$ is a new vertex (at which nothing was previously marked), if the degree of $v$ is at most $k_0$
		and if the \inps\ are identical in $\wpg_{t_0}$ and $\apg_{t_0}$,
		we consider the remaining half-edges at $v$ and apply the rules of Warning Propagation to determine whether
		any \outmessages\ will change. Any that do become marked. We call such a vertex a \emph{standard vertex}.
		\item If $v$ does not satisfy all three of these conditions, we say that we have \emph{hit a \mine}.
		In particular:
		\begin{itemize}
			\item If $v$ is a vertex that we have seen before, it is called a \emph{\landmine} vertex;
			\item If $v$ is a vertex of degree at most $d_0$ whose \inps\ are different
			according to $\wpg_{t_0}$ and $\apg_{t_0}$, it is called an \emph{\seamine} vertex;
			\footnote{Note that error vertices include in particular those at which we deleted unmatched half-edges
			in Step~\ref{def:apgMatch} of the construction of $\apg_{t_0}$.}
			\item If $v$ is a vertex of degree larger than $d_0$, it is called a \emph{\nuclearmine} vertex.
		\end{itemize}
		In each case, all of the half-edges at $v$ become marked. Such half-edges are called \emph{\shrapnel} edges,
		and are further classified as \emph{\rubble}, \emph{\spray} and \emph{\fallout} respectively,
		according to the type of \mine\ we hit.
		The corresponding messages
		can change arbitrarily (provided each individual change is in $\potentialchanges\bc{P}$).
	\end{itemize}
\end{itemize}

Note that a \landmine\ vertex may also be either an \seamine\ or a \nuclearmine\ vertex. However,
by definition, no \mine\ is both an \seamine\ and a \nuclearmine\ vertex.

We first justify that this gives an upper bound on the number of changes made by Warning Propagation.
Let $\WPchanges$ be the set of edges on which the messages are different in $\proj{\wpg}_{t_0}$ and
in $\WP^*\bc{\proj{\wpg}_{t_0}}$, and let $\markchanges$ be the set of edges which are marked at the end of the marking process.
Note that the set $\markchanges$ is not uniquely defined, but depends on the arbitrary
choices for the changes which are made at \mines.

\begin{proposition}\label{prop:changescontainment}
	There exists some choice of the changes to be made at \mines\ such that
	$
	\WPchanges \subset \markchanges.
	$
\end{proposition}

\begin{proof}
	We proceed in rounds indexed by $t \in \NN_0$. We define
	$\WPchanges\bc{t}$ to be the set of edges on which the messages are different in $\WP^t\bc{\proj{\wpg}_{t_0}}$ compared
	to $\proj{\wpg}_{t_0}$, while $\markchanges\bc{t}$ is the set of edges which are marked after $t$ steps of the marking process.
	Since $\WPchanges = \lim_{t\to \infty}\WPchanges\bc{t}$ and $\markchanges = \lim_{t\to \infty} \markchanges\bc{t}$,
	it is enough to prove that for each $t \in \NN_0$ we have $\WPchanges\bc{t} \subset \markchanges \bc{t}$,
	which we do by induction on $t$.
	
	The base case $t=0$ is simply the statement that the set of initial marks contains the changes from $\proj{\wpg}_{t_0}$ to
	$\proj{\wpg}_{t_0+1}$,
	which is clearly true by construction.
	
	For the inductive step, each time we reveal the incoming partner of a marked outgoing half-edge,
	if this is a vertex at which nothing was previously marked, i.e.\ a standard vertex, then we proceed
	with marking exactly according to Warning Propagation.
	
	On the other hand, if at least one edge was already marked at this vertex we simply mark \emph{all} the outgoing half-edges,
	and if we choose the corresponding changes according to the changes that will be made by Warning Propagation, the induction
	continues.
\end{proof}

In view of Proposition~\ref{prop:changescontainment}, our main goal is now the following.

\begin{lemma}\label{lem:fewmarked}
	At the end of the marking process, \whp\ at most $\sqrt{\delta_0}n$ edges are marked.
\end{lemma}

During the proof of Lemma~\ref{lem:fewmarked}, we will make extensive use of the following fact.

\begin{claim}\label{claim:manyhistories}
	\Whp, for every $\vom \in \alphabet^{t_0+1}$ such that $ \Pr_{\viteratdist{t}}\bc{ \vom} \neq 0$,
	the total number of \inps\ of $\vom$
	over all vertices is at least $\delta_0^{1/100}n$.
\end{claim}

\begin{proof}
Since $\Pr_{\viteratdist{t}}\bc{ \vom} \neq 0$, there certainly exists some $d \in \NN$ and some $A \in \ms{\alphabet^{t_0+1}}{d}$
such that $\vom \in A$ and $\gamma_A>0$.
	Since we chose $\delta_0$ sufficiently small, so in particular $\delta_0^{1/100} < \gamma_A$,
	Proposition~\ref{prop:incominghistoryconcentration} implies that \whp\ there are certainly at least
	$\gamma_{A} n-o \bc{n} \ge \delta_0^{1/100}n$  vertices which receive \inp\ $A$,
	which is clearly sufficient.
\end{proof}

Given a positive real number $d$ and a probability distribution $\cD$ on $\NN_0^k$, we denote by $\conddist{\cD}{\leq d }$ the probability distribution $\cD$ conditioned on the event $\norm{\cD}_1 \leq d$.  
Recall that $P':=\distf^{t_0-1}\bc{\initialdistribution}$, and
recall also from Definition~\ref{def:GenMulti} that $\MSet{\cD}{\vec{q}}$
is a random multiset of messages. With a slight abuse of notation,
we will also use $\MSet{\cD}{\vec{q}}$ to refer to the \emph{distribution} of this random multiset.
\begin{proposition}
	Whenever a standard vertex $v$ is revealed in the marking process from a change of type $\changeboth{1}$, the further changes
	made at outgoing half-edges at $v$ have asymptotically the same distribution as in the branching process $\changebranch\bc{\changeold{1},\changenew{1},P'}$
	below a change of type $\changeboth{1}$.
\end{proposition}

\begin{proof}		
		First, we note that $v$ is revealed in the marking process from a change of type $\changeboth{1}$
		so the vertex $v$ has type $i:=g_1\bc{\changeold{1}}$ and its parent
		(i.e its immediate predecessor in the branching process $\changebranch\bc{\changeold{1},\changenew{1},P'}$) has type $j:=g_2\bc{\changeold{1}}$.
		Now, given that $v$ is a standard vertex , we may use $\apg_{t_0}$ instead of $\wpg_{t_0}$ to model it.
		Moreover, there are $\offdist{j}{i}\vert_{\leq d_0}$ further half-edges at $v$. By Remark~\ref{rem:finitemoments} and Markov's inequality, the event $\norm{\offdist{j}{i}}_1 \leq d_0$ is a high probability event. Thus,  the  distribution $\offdist{j}{i}\vert_{\leq d_0}$ tends asymptotically
		to the  distribution $\offdist{j}{i}$.
		Furthermore, by Claim~\ref{claim:localhistorydistribution}, each of these further half-edges has a
		\inmessaget{t_0} distributed according to $P'$ independently.
		Since $v$ was a new vertex, these \inmessages\ have not changed,
		and therefore are simply distributed according
		to $\MSet{\offdist{j}{i}}{P'[i]}$, as in $\changebranch\bc{\changeold{0},\changenew{0},P'}$.
		
		Note that in the idealised process $\changebranch\bc{\changeold{0},\changenew{0}, P'}$ we additionally
		condition on these incoming messages producing~$\changeup{0}$, the appropriate message to the parent.
		In this case we do not know the message that $v$ sent to its ``parent'',
		in the marking process. However, this message is distributed as $P'[i,j]$,
		and letting $X$ denote a random variable distributed as  $\MSet{\offdist{j}{i}}{\vec{q}_i}$,
		the probability that the multiset of incoming messages at $v$ is $A$ is simply
		\begin{align*}
		\Pr\bc{P'[i,j]=\wpf\bc{A}} \Pr\bc{X=A \; | \; \wpf\bc{X}=\wpf\bc{A}}.
		\end{align*}
		Since $P'$ is asymptotically close to the stable fixed point $P$, we have that
		$\Pr\bc{P'[i,j]=\wpf \bc{A}}$ is asymptotically close to $\Pr\bc{\wpf\bc{X}=\wpf\bc{A}}$ for each $A$,
		and so the expression above can be approximated simply by $\Pr\bc{\cbc{X=A} \cap \cbc{\wpf\bc{X}=\wpf\bc{A}}}$ \ $= \Pr\bc{X=A}$, as required.
\end{proof}

\subsection{Three stopping conditions}

In order to prove Lemma~\ref{lem:fewmarked}, we introduce some stopping conditions on the marking process.
More precisely, we will run the marking process until one of the following three conditions is satisfied.

\begin{enumerate}
	\item \emph{Exhaustion} - the process has finished.
	\item \emph{Expansion} - there exists some $\changeboth{1} = \bc{\changeold{1},\changenew{1}} \in \alphabet^2$
	such that at least $\delta_0^{3/5}\alpha_{\changeboth{1}}n$ messages
	have changed from $\changeold{1}$ to $\changenew{1}$ (where $\vec \alpha$ is the vector from Corollary~\ref{cor:subeigenvector2}).
	\item \emph{Explosion} - the number of \shrapnel\ edges is at least $\delta_0^{2/3}n$.
\end{enumerate}

Lemma~\ref{lem:fewmarked} will follow if we can show that \whp\ neither expansion nor explosion occurs.

\subsubsection{Explosion}

\begin{proposition}\label{prop:noexplosion}
	\Whp\ explosion does not occur.
\end{proposition}

We will split the proof up into three claims, dealing with the three different types of \shrapnel\ edges.

\begin{claim}\label{claim:rubblebound}
	\Whp, the number of \rubble\ edges is at most $\delta_0^{2/3}n/2$.
\end{claim}

	\begin{proof}
		A type-$i$ vertex $v$ of degree  $d$ will contribute $d$ \rubble\ edges
		if it is chosen at least twice as the partner of a marked half-edge. Using Claim~\ref{claim:manyhistories}, at each step there
		are at least $\delta_0^{1/100}n$  possible half-edges to choose from, of which certainly at most $d$ are incident to $v$,
		and thus the probability that $v$ is chosen twice in the at most $\sqrt{\delta_0} n$ steps is at most
		$$
		\bc{\frac{d}{\delta_0^{1/100}n}}^2 \bc{\sqrt{\delta_0}n}^2 = \delta_0^{49/50} d^2.
		$$
		Thus setting $S$ to be the number of \rubble\ edges and $c:=\max_{i \in [k]}\Erw\bc{ \norm{\degdist_i}_1^3 }$, we have
		\begin{align*}
		\Erw\bc{S} \le \sum_{i=1}^k \sum_{d=0}^\infty d \bc{\Pr\bc{\norm{\degdist_i}_1=d} n_i } \delta_0^{49/50} d^2 
		& = \delta_0^{49/50} \sum_{i=1}^k n_i  \sum_{d=0}^\infty d^3 \Pr\bc{\norm{\degdist_i}_1=d} \\
	 	&\le \delta_0^{49/50} \cdot\delta_0^{-1/100}n \cdot c  \le \delta_0^{4/5} n.
		\end{align*}
		On the other hand, if two distinct vertices have degrees~$d_1$ and~$d_2$, then the probability that both become \mines\ 
		may be estimated according to whether or not they are adjacent to each other, and is at most
		$$
		\frac{d_1d_2}{\delta_0^{1/100}n} \cdot \frac{d_1d_2}{\bc{\delta_0^{1/100}n}^3}\bc{\sqrt{\delta_0}n}^3 + \frac{d_1^2 d_2^2}{\bc{\delta_0^{1/100}n}^4}\bc{\sqrt{\delta_0} n}^4
		\le 2d_1^2 d_2^2 \delta_0^{24/25}.
		$$
		Therefore we have
		\begin{align*}
		\Erw\bc{S^2} & \le \Erw\bc{S} + \sum_{i,j,\ell,m \in [k]} \sum_{d_1,d_2=0}^\infty d_1d_2 \Pr\bc{\norm{\offdist{j}{i}}_1=d_1} n_i \cdot \Pr\bc{\norm{\offdist{\ell}{m}}_1=d_2} n_m
		\cdot  2d_1^2 d_2^2\delta_0^{49/25}\\
		& \le \delta_0^{4/5}n + 2\delta_0^{49/25} \max_{i,j\in [k]}\bc{\Erw\bc{\norm{\offdist{j}{i}}_1^3}}^2  \bc{\delta_0^{-1/100} n}^2 \\
		& \le \delta_0^{4/5}n + \delta_0^{48/25} n^2 \max_{i,j\in [k]}\bc{\Erw\bc{\norm{\offdist{j}{i}}_1^3}}^2
		\le \delta_0^{9/5}n^2,
		\end{align*}
		where the last line follows due to Remark~\ref{rem:finitemoments} for sufficiently small $\delta_0$.
		Finally, Chebyshev's inequality shows that \whp\ the number of \shrapnel\  is at most $\delta_0^{2/3} n/2$, as claimed.
\end{proof}

Recall that $a_0 := \frac{\sqrt{c_0}}{4d_0\abs{\alphabet}^{\bc{t_0+2}d_0}}$.
\begin{claim}\label{claim:spraybound}
	\Whp\ the number of \spray\ edges is at most $\frac{d_0 n}{\sqrt{a_0}}$. 
\end{claim}

\begin{proof}
	Observe that Corollary~\ref{cor:similarhistories} implies in particular
	that the number of edges of $\wpg_{t_0}$ which are attached to vertices of degree at most $d_0$ where the incoming message
	histories differ from those in $\apg_{t_0}$ (i.e.\ which would lead us to an \seamine\ vertex if chosen) is at most
	$
	d_0 \frac{n}{a_0}, 
	$
	and therefore the probability that we hit an \seamine\ in any one step is at most
	$
	\frac{d_0 n/a_0}{\delta_0^{1/100}n} = \frac{1}{\delta_0^{1/100} \bc{a_0/d_0}}.
	$
	Furthermore, any time we meet an \seamine\,
	we obtain at most $d_0$ \spray\ edges, and since the marking process proceeds for at most $\delta_0^{3/5}n$ steps,
	therefore the expected number of \spray\ edges in total is at most
$$
\delta_0^{3/5}n \cdot \frac{d_0}{\delta_0^{1/100} \bc{a_0/d_0}}= \delta_0^{59/100}n \cdot \frac{d_0^2}{a_0}.
$$
Now, by \eqref{PP:choicec0}, we have $c_0 \gg \exp\bc{Cd_0} \gg d_0^6 \abs{\alphabet}^{2 \bc{t_0+2} d_0}$
so $\sqrt{ c_0} \gg d_0^3 \abs{\alphabet}^{ \bc{t_0+2} d_0}$ which implies that $a_0 \gg d_0^2$.
Thus, application of Markov's inequality completes the proof.
\end{proof}

\begin{claim}\label{claim:falloutbound}
	\Whp the number of \fallout\ edges is at most $\Delta_0 \frac{n}{\sqrt{c_0}}$.
\end{claim}

\begin{proof}
	This is similar to the proof of Claim~\ref{claim:spraybound}.
	By assumption~\ref{assump:Delta}, \whp\ there are no vertices of degree larger than $\Delta_0$.
	Moreover, by Proposition~\ref{prop:incominghistoryconcentration},
	\whp\ the number of edges adjacent to vertices of degree at least $d_0$ is at most
	$n/c_0$, so the probability of hitting a \nuclearmine\ is at most $\frac{\Delta_0}{c_0}$.
	If we hit a \nuclearmine, 
	at most $\Delta_0$ half-edges become \fallout,
	therefore the expected number of \fallout\ edges is  $\delta_0^{3/5}n \cdot O\bc{\Delta_0 \cdot \frac{\Delta_0}{c_0}} = O\bc{\frac{\Delta_0^2 n}{c_0}}$.
By~\ref{PP:choicec0} we have $c_0 \gg \Delta^2$ so an application of Markov's inequality completes the proof.
\end{proof}

Combining all three cases we can prove Proposition~\ref{prop:noexplosion}.

\begin{proof}[Proof of Proposition~\ref{prop:noexplosion}]
	By Claims~\ref{claim:rubblebound},~\ref{claim:spraybound} and~\ref{claim:falloutbound},
	\whp\ the total number of \shrapnel\ edges is at most
	$$ 
		\frac{\delta_0^{2/3}n}{2} + \frac{d_0  n}{\sqrt{a_0}} + \frac{\Delta_0n}{\sqrt{c_0}} 
		$$
	Again, by \eqref{PP:choicec0}, we have $c_0 \gg \exp\bc{Cd_0} \gg d_0^6 \abs{\alphabet}^{2\bc{t_0+2} d_0}$ and $c_0 \gg \Delta_0^2$. Thus, we have  $\sqrt{a_0} \gg d_0$  and $\sqrt{c_0} \gg \Delta_0$. Hence, 
		$$ 
		\frac{\delta_0^{2/3}n}{2} + \frac{d_0 n}{ \sqrt{a_0} } + \frac{ \Delta_0 n}{ \sqrt{c_0}} \le \delta_0^{2/3} n
		$$
		as required.	
\end{proof}

\subsubsection{Expansion}

\begin{proposition}\label{prop:noexpansion}
	\Whp\ expansion does not occur.
\end{proposition}

\begin{proof}
	By Proposition~\ref{prop:noexplosion}, we may assume that explosion does not occur, so we have few \shrapnel\ edges.
	Therefore in order to achieve expansion, at least $\frac{2}{3}\sqrt{\delta_0}n$ edges would have to be marked
	in the normal way, i.e.\ by being generated as part of a $\changebranch$ branching process rather than as one of the $\delta_0 n$ initial
	half-edges or as a result of hitting a \mine.

	However, we certainly reveal children in $\changebranch$
	of at most $\delta_0^{3/5}\alpha_{\changeboth{2}} n$ 
	changes from $\changeold{2}$ to $\changenew{2}$, for each choice of $\changeboth{2}=\bc{\changeold{2},\changenew{2}} \in \alphabet^2$,
	since at this point the expansion stopping condition would be applied.
	Thus the expected number of changes from $\changeold{1}$ to $\changenew{1}$ is at most
	$$
	\sum_{\changeboth{2} \in \alphabet} \delta_0^{3/5} \alpha_{\changeboth{2}} n T_{\changeboth{1},\changeboth{2}}
	= \bc{T\alpha}_{\changeboth{1}} \delta_0^{3/5} n \le \bc{1-\gamma} \alpha_{\changeboth{1}} \delta_0^{3/5} n.
	$$
	We now aim to show that \whp\ the actual number of changes is not much larger than this (upper bound on the) expectation,
	for which we use a second moment argument.
	Let us fix some $\changeboth{2} \in \alphabet^2$.
	For simplicity, we will assume for an upper bound that we reveal precisely $s:=\delta_0^{3/5} \alpha_{\changeboth{2}}n$ changes of type $\changeboth{2}$
	in $\changebranch$.
	Then the number of changes of type $\changeboth{1}$ that arise from these
	is the sum of $s$ independent and identically distributed integer-valued random variables
	$X_1,\ldots,X_s$, where for each $r\in [s]$ we have
	$\Erw\bc{X_r} =T_{\changeboth{1},\changeboth{2}}$
	and $\Erw\bc{X_r^2} \le c:=\max_{i,j\in [k]}\Erw\bc{\norm{\offdist{j}{i}}_1^2} $.
	Therefore we have $\Var\bc{X_r} \le c^2 = O\bc{1}$, and the central limit theorem tells us that
	$\Var\bc{\sum_{r=1}^s X_r} = O\bc{\sqrt{s}}$.
	Then the Chernoff bound implies that \whp
	$$
	\abs{\sum_{r=1}^s X_r - \Erw\bc{\sum_{r=1}^s X_r}} \le n^{1/4} O\bc{\sqrt{s}}
	= O\bc{n^{3/4}} \le \frac{\gamma}{2} \delta_0^{3/5}T_{\changeboth{1},\changeboth{2}} \alpha_{\changeboth{2}}n.
	$$
	Taking a union bound over all $\abs{\alphabet}^4$ choices of $\changeboth{1},\changeboth{2}$, we deduce that \whp\ the total
	number of changes of type $\changeboth{1}$ is at most
	$$
	\bc{1-\gamma} \alpha_{\changeboth{1}} \delta_0^{3/5} n + \sum_{\changeboth{2}}\frac{\gamma}{2} \delta_0^{3/5}T_{\changeboth{1},\changeboth{2}} \alpha_{\changeboth{2}}n
	=\bc{1-\gamma/2} \alpha_{\changeboth{1}} \delta_0^{3/5} n
	$$
	for any choice of $\changeboth{1}$, as required.
\end{proof}

\subsubsection{Exhaustion}

\begin{proof}[Proof of Lemma~\ref{lem:fewmarked}]
	By Propositions~\ref{prop:noexplosion} and~\ref{prop:noexpansion}, neither explosion nor expansion occurs.
	Thus the process finishes with exhaustion, and (using the fact that $\|\vec\alpha\|_1=1$) the total number of edges marked is at most
	$$
	\sum_{\changeboth{1}\in \alphabet^2} \delta_0^{3/5} \alpha_{\changeboth{1}}n + \delta_0^{2/3}n = \bc{\delta_0^{3/5}+\delta_0^{2/3}}n \le \sqrt{\delta_0}n
	$$
	as required.
\end{proof}

\subsection{Proof of Theorem~\ref{thm:main}}

We can now complete the proof of our main theorem.

\begin{proof}[Proof of Theorem~\ref{thm:main}]
	Recall from Proposition~\ref{prop:changescontainment} that edges on which
	messages change when moving from
	$\WP^{t_0}\bc{\wpg_0}$ to $\WP^{*}\bc{\wpg_0}$, which are simply those in the set $\WPchanges$,
	are contained in $\markchanges$.
	
	Furthermore, Lemma~\ref{lem:fewmarked} states that $\abs{\markchanges} \le \sqrt{\delta_0}n$.
	Since we chose $\delta_0 \muchless \delta$, the statement of
	Theorem~\ref{thm:main} follows.
\end{proof}

\section{Concluding remarks}

We remark that in the definition of the~$\apg_{t_0}$ model,
rather than deleting unmatched half-edges, an alternative approach
would be to condition on the event that the statistics match up in such a way that no
half-edges need be deleted, i.e.\ such that the number of half-edges
with \instoryt{t_0} $\vec \mu_1$ and \outstoryt{t_0} $\vec \mu_2$
is identical to the number of half-edges with \instoryt{t_0} $\vec \mu_2$
and \outstoryt{t_0} $\vec \mu_1$, while the number of half-edges with both
\instoryt{t_0} and \outstoryt{t_0} $\vec \mu$ is even. Subsequently one would need to show that
this conditioning does not skew the distribution too much, for which it ultimately suffices
to show that the event has a probability of at least $n^{-\Theta\bc{1}}$.

In some ways this might even be considered the more natural approach,
and indeed it was the approach we initially adopted in early versions of this paper.
However, while the statement that the conditioning event is at least polynomially likely
is an intuitively natural one when one considers that, heuristically,
the number of half-edges with each story should be approximately normally distributed
with standard deviation $O\bc{\sqrt{n}}$,
proving this formally is surprisingly delicate and involves some significant
technical difficulties.

Since at other points in the proof we already need to deal with ``errors'',
and unmatched half-edges can be handled as a subset of these,
this approach turns out to be far simpler and more convenient.

\section{Acknowledgement}

We are very grateful to Amin Coja-Oghlan and Mihyun Kang for their helpful contributions
to an earlier version of this project.


\begin{thebibliography}{29}

\bibitem{Achlioptas}
D.~Achlioptas: Lower Bounds for Random 3-SAT via Differential Equations. Theoretical Computer Science {\bf265} (2001) 159--185.

\bibitem{Barriers}
D.~Achlioptas, A.~Coja-Oghlan:
Algorithmic barriers from phase transitions.
Proc.~49th FOCS (2008) 793--802.

\bibitem{AchlioptasMolloy}
D.\ Achlioptas, M.\ Molloy: The solution space geometry of random linear equations.
Random Structures and Algorithms {\bf46} (2015) 197--231.



\bibitem{we}
A.~Coja-Oghlan, O.~Cooley, M.~Kang, J.~Lee, J.B.~Ravelomanana: The sparse parity matrix. ArXiv 2107.06123

\bibitem{COCKS2}
A.~Coja-Oghlan, O.~Cooley, M.~Kang, K.~Skubch:
Core forging and local limit theorems for the k-core of random graphs.
J.\ Comb.\ Theory, Ser.\ B {\bf 137} (2019) 178--231.

\bibitem{COFKR}
A.~Coja-Oghlan, U.~Feige, M.~Krivelevich, D.~Reichman:
Contagious Sets in Expanders. Proc.\ 26th SODA (2015) 1953--1987.


\bibitem{CKZ21}
O.~Cooley, M.~Kang, J.~Zalla:
Loose cores and cycles in random hypergraphs. ArXiv 2101.05008.

\bibitem{Cooper}
C.\ Cooper: The cores of random hypergraphs with a given degree sequence.
Random Structures and Algorithms {\bf25} (2004) 353--375.

\bibitem{Darling}
R.\ Darling, J.\ Norris:
Differential equation approximations for Markov chains.
Probability Surveys {\bf 5} (2008) 37--79.

\bibitem{DuboisMandler}
O.\ Dubois, J.\ Mandler: The 3-XORSAT threshold. Proc.\  43rd FOCS (2002) 769--778.



\bibitem{Fernholz1}
D.\ Fernholz, V.\ Ramachandran: The giant $k$-core of a random graph with a specified degree sequence.
Manuscript (2003).

\bibitem{Fernholz2}
D.\ Fernholz, V.\ Ramachandran:
Cores and connectivity in sparse random graphs.
UTCS Technical Report TR04-13 (2004).

\bibitem{FS}
A.~Frieze, S.~Suen: Analysis of Two Simple Heuristics on a Random Instance of k-SAT. J.\ Algorithms {\bf20} (1996) 312--355. 

\bibitem{Gallager}
R.\ Gallager: Low-density parity check codes.
IRE Trans.\ Inform.\ Theory {\bf 8} (1962) 21--28.

\bibitem{Ibrahimi}
M.\ Ibrahimi, Y.\ Kanoria, M.\ Kraning, A.\ Montanari:
The set of solutions of random XORSAT formulae. 
Ann.\ Appl.\ Probab.\ {\bf 25} (2015) 2743--2808.

\bibitem{JansonLuczak}
S.\ Janson, M.\ Luczak: A simple solution to the $k$-core problem.
Random Structures and Algorithms {\bf 30} (2007) 50--62.

\bibitem{JansonLuczak2}
S.\ Janson, M.\ Luczak: Asymptotic normality of the $k$-core in random graphs.
Ann.\ Appl.\ Probab.\ {\bf 18} (2008) 1085--1137.


\bibitem{Kim}
J.H.~Kim:
Poisson cloning model for random graphs.
Proceedings of the International Congress of Mathematicians  (2006)  873--897.

\bibitem{MM}
M.~M\'ezard, A.~Montanari:
Information, physics and computation.
Oxford University Press~2009.

\bibitem{MolloyCores}
M.\ Molloy: Cores in random hypergraphs and Boolean formulas.
Random Structures and Algorithms {\bf27}  (2005) 124--135.

\bibitem{Molloy}
M.~Molloy: The freezing threshold for $k$-colourings of a random graph.
J.\ ACM {\bf65} (2018) \#7.

\bibitem{MR}
M.~Molloy, R.~Restrepo: Frozen variables in random boolean constraint satisfaction problems.
Proc.\ 24th SODA (2013) 1306--1318.

\bibitem{Pittel}
B.\ Pittel, J.\ Spencer, N.\ Wormald:
Sudden emergence of a giant $k$-core in a random graph.
Journal of Combinatorial Theory, Series B {\bf 67} (1996) 111--151

\bibitem{RichardsonUrbanke}
T.\ Richardson, R.\ Urbanke: Modern coding theory. Cambridge University Press (2008).

\bibitem{Riordan}
O.\ Riordan:
The $k$-core and branching processes.
Combinatorics, Probability and Computing {\bf 17} (2008) 111--136.

\bibitem{Skubch}
K.~Skubch:
The core in random hypergraphs and local weak convergence. ArXiv 1511.02048.


\bibitem{Wormald}
N.~Wormald: Differential equations for random processes and random graphs.  Ann.\ Appl.\ Probab.\ {\bf5} (1995) 1217--1235.


\end{thebibliography}
\end{document}